\documentclass{amsart}

\usepackage{amsfonts}
\usepackage{amsmath}
\usepackage{amsthm}
\usepackage[T1]{fontenc}
\usepackage[utf8]{inputenc}
\usepackage{latexsym}
\usepackage{longtable}
\usepackage{mathrsfs}
\usepackage{mathtools}
\usepackage{slashed}
\usepackage{stmaryrd}
\usepackage{tikz}
    \usetikzlibrary{matrix,arrows}
\usepackage{xcolor}
\usepackage[all]{xy}
\usepackage{zhlipsum}
\usepackage{hyperref}
    \hypersetup{
        colorlinks,
	citecolor=black,
	filecolor=black,
	linkcolor=black,
	urlcolor=black
    }

\newtheorem{theorem}{Theorem}[section]
\newtheorem{lemma}[theorem]{Lemma}
\newtheorem{proposition}[theorem]{Proposition}
\newtheorem{corollary}[theorem]{Corollary}

\theoremstyle{definition}
\newtheorem{definition}[theorem]{Definition}

\newtheorem{thm}{\bf Theorem}

\theoremstyle{remark}
\newtheorem{remark}[theorem]{Remark}

\numberwithin{equation}{section}

\def  \Ab       {\mathrm{Ab}}
\def  \acn      {\mathrm{acn}}
\def  \ad       {\mathrm{ad}}

\def  \aperf    {\mathrm{aperf}}

\def  \Aut      {\mathrm{Aut}}
\def  \bDelta   {{\boldsymbol{\Delta}}}

\def  \CAlg     {\mathrm{CAlg}}

\def  \CDiv     {\mathrm{CDiv}}
\def  \cl       {\mathrm{cl}}

\def  \cn       {\mathrm{cn}}

\def  \cpl      {\mathrm{cpl}}

\def  \cts      {\mathrm{c}}

\def  \et       {\mathrm{\acute{e}t}}

\def  \Def      {\mathrm{Def}}
\def  \Der      {\mathrm{Der}}

\def  \dr       {\mathrm{Dr}}
\def  \Dr       {{\mathcal{D}r}}
\def  \einf     {\be_\infty}
\def  \El       {\mathrm{Ell}}

\def  \Ext      {\mathrm{Ext}}
\def  \FFG      {\mathrm{FFG}}
\def  \fg       {\mathrm{FG}}
\def  \fib      {\mathrm{fib}}
\def  \fin      {\mathrm{fin}}
\def  \fm       {\mathfrak{m}}
\def  \Frob     {\mathrm{Frob}}
\def  \Fun      {\mathrm{Fun}}
\def  \Gal      {\mathrm{Gal}}
\def  \GL       {\mathrm{GL}}

\def  \GrpCDiv  {\mathrm{GrpCDiv}}
\def  \h        {\mathrm{h}}
\def  \Hilb     {\mathrm{Hilb}}
\def  \hinf     {\bh_\infty}

\def  \Hom      {\mathrm{Hom}}
\def  \id       {\mathrm{id}}

\def  \IndCoh   {\mathrm{IndCoh}}

\def  \jl       {\mathcal{J\!L}}

\def  \Level    {\mathrm{Level}}

\def  \ll       {{^\mathrm{L}}}
\def  \LMod     {\mathrm{LMod}}

\def  \lt       {\mathrm{LT}}
\def  \LT       {\mathcal{LT}}
\def  \Map      {\mathrm{Map}}

\def  \Mod      {\mathrm{Mod}}

\def  \op       {\mathrm{op}}
\def  \Or       {\mathrm{or}}

\def  \QCoh     {\mathrm{QCoh}}

\def  \ringtop  {{\infty\Top^{\mathrm{loc}}_\CAlg}}
\def  \Sch      {\mathrm{Sch}}
\def  \Set      {\mathrm{Set}}
\def  \Shv      {\mathcal{S}\mathrm{hv}}

\def  \SpDM     {\mathrm{SpDM}}
\def  \Spec     {\mathrm{Spec}\,}
\def  \Spet     {\mathrm{Sp\acute{e}t}\,}
\def  \Spf      {\mathrm{Spf}\,}
\def  \spic     {{\mathscr{P}\mathrm{ic}}}

\def  \taq      {\mathrm{TAQ}} 
\def  \TMF      {\mathrm{TMF}} 
\def  \Top      {\mathrm{Top}} 
\def  \Tor      {\mathrm{Tor}}

\def  \Tot      {\mathrm{Tot}}
\def  \tr       {\mathrm{tr}}
\def  \un       {\mathrm{un}}
\def  \univ     {\mathrm{univ}}
\def  \Var      {\mathrm{Var}}

\def  \cd       {\mathcal{D}}

\def  \ch       {\mathcal{H}}
\def  \ci       {\mathcal{I}}

\def  \cm       {\mathcal{M}}
\def  \co       {\!:}
\def  \cO       {\mathcal{O}}
\def  \cp       {\mathcal{P}}
\def  \cs       {\mathcal{S}}

\def  \cx       {\mathcal{X}}
\def  \cy       {\mathcal{Y}}

\def  \ba       {\mathbf{A}}
\def  \bc       {\mathbf{C}}
\def  \be       {\mathbb{E}}
\def  \bF       {\mathbf{F}}
\def  \bg       {\mathbb{G}}
\def  \bG       {\mathbf{G}}
\def  \hG       {\widehat{\bG}}
\def  \bh       {\mathbb{H}}
\def  \hH       {\widehat{\mathbf{H}}}

\def  \bp       {\mathbf{P}}
\def  \bq       {\mathbf{Q}}

\def  \bs       {\mathbf{S}}

\def  \bz       {\mathbf{Z}}

\def  \sd       {\mathsf{D}}
\def  \se       {\mathsf{E}}
\def  \sE       {\mathscr{E}}
\def  \sf       {\mathscr{F}}

\def  \so       {\mathscr{O}}

\def  \ss       {\mathsf{S}}
\def  \st       {\mathsf{T}}
\def  \sx       {\mathsf{X}}
\def  \sy       {\mathsf{Y}}
\def  \sz       {\mathsf{Z}}

\begin{document}

\title[Spectral moduli problems for level structures]{
Spectral moduli problems for level structures and an integral Jacquet--Langlands dual of Morava E-theory 
}

\author{Xuecai Ma}
\address{
    Westlake Institute for Advanced Study,  Hangzhou, Zhejiang 310024, P.R. China \textcolor{white}{.} \\
    \hskip-.3cm\textcolor{white}{\raisebox{-.09cm}{$\bullet$}\hskip.04cm} Institute for Theoretical Sciences, Westlake University, Hangzhou, Zhejiang 310030, P.R. China
}
\curraddr{}
\email{maxuecai@westlake.edu.cn}
\thanks{}

\author{Yifei Zhu}
\address{Department of Mathematics, College of Science, Southern University of Science and Technology, Shenzhen, Guangdong 518055, P.R. China}
\email{zhuyf@sustech.edu.cn}


\begin{abstract}
    Given an $\einf$-ring spectrum $R$, with motivation from chromatic homotopy theory, we define relative effective Cartier divisors for a spectral Deligne--Mumford stack over $\Spet R$ and prove that, as a functor from connective $R$-algebras to topological spaces, it is representable.  This enables us to solve various moduli problems of level structures on spectral abelian varieties, overcoming difficulty at primes dividing the level.  In particular, we obtain higher-homotopical refinement for finite levels of a Lubin--Tate tower as $\einf$-ring spectra, which generalizes Morava, Hopkins, Miller, Goerss, and Lurie's spectral realization of the deformation ring at the ground level.  Moreover, passing to the infinite level and then descending along the equivariantly isomorphic Drinfeld tower, we obtain a Jacquet--Langlands dual to the Morava E-theory spectrum, along with homotopy fixed point spectral sequences dual to those studied by Devinatz and Hopkins.  These serve as potential tools for computing higher-periodic homotopy types from pro-\'etale cohomology of $p$-adic general linear groups.  
    
\end{abstract}

\maketitle

\tableofcontents

\section{Introduction}

The stable homotopy category is a central object of study in algebraic topology, and chromatic homotopy theory provides it with an organizing principle through the hight filtration for the moduli stack of one-dimensional formal groups.  Structured ring spectra are the most common examples therein, such as $\hinf$ and $\einf$ structures encoding homotopy coherence.  In \cite{lurie2009survey,lu-EC2}, Lurie uses methods from spectral algebraic geometry, i.e., algebraic geometry over the sphere spectrum $\bs$, to give a proof for the Goerss--Hopkins--Miller theorem of topological modular forms.  Besides applications to elliptic cohomology, Lurie also proves the $\einf$ structures on Morava E-theory spectra \cite{lu-EC2}, which relies on a spectral version of deformation theory for certain $p$-divisible groups.  

The earlier proof of $\einf$ structures on E-theories was due to Goerss, Hopkins, and Miller \cite{goerss2004moduli}.  They turned the problem into a moduli problem of realizations for commutative algebras in certain comodules as $\einf$-ring spectra, and developed an obstruction theory, completing the proof by computing Andr\'e–Quillen cohomology groups.  Compared with their methods, Lurie's proof is more conceptual, realizing the chromatic point of view with suitable tools from higher category theory.  There have been more and more applications of spectral algebraic geometry in algebraic topology, such as topological automorphic forms \cite{behrens2010topological}, Morava E-theories over any $\bF_p$-algebra (not just perfect fields of characteristic $p$) \cite{lu-EC2}, equivariant topological modular forms \cite{gepner2020equivariant,TJF}, and elliptic Hochschild homology \cite{sibilla2023equivariant}, among others.  (Meanwhile, the Goerss--Hopkins obstruction theory has also been generalized to higher categorical settings \cite{PVK,MG}, but see the historical footnote 3 of \cite{MG}.)  

In classical algebraic geometry, moduli problems concerning deformations of formal groups with level structure are well studied.  In particular, moduli spaces for all levels form a Lubin--Tate tower \cite{rapoport1996period,fargues2008isomorphisme,scholze2013moduli}.  We know from the Goerss--Hopkins--Miller--Lurie theorem that the universal objects of deformations of formal groups have higher-homotopical analogues which are the Morava E-theories.  A natural question is then what higher-homotopical analogues can be for moduli problems of deformations {\em with level structure}.  Moreover, can we find such analogues for Lubin--Tate towers?  Although $\einf$ structures on topological modular forms with level structure have been studied in \cite{hill2016topological}, even functorially over the compactified moduli stacks through logarithmic structures, we aim to find derived stacks of spectral elliptic curves with level structure that allow the level to not be inverted therein, as opposed to the properties of \'etale or log-\'etale required by earlier methods using (logarithmic) obstruction theory.  Besides, in the computation of chromatic unstable homotopy groups of spheres, after applying the EHP sequence and Bousfield--Kuhn functor, we observe that certain terms on the $E_2$-page of a homotopy fixed point spectral sequence arise from universal deformations of isogenies of formal groups.  These are computed through the Morava E-theory on classifying spaces of symmetric groups \cite{strickland1997finite,strickland1998morava}.  They can be suitably interpreted as sheaves on a Lubin--Tate tower.  We aim to provide a more conceptual perspective on this fact via the proposed higher-homotopical Lubin--Tate tower and its geometry.  

In this paper, we address these questions by studying specific moduli problems in spectral algebraic geometry.  A main ingredient of our work is the derived analogue of Artin's representability theorem established in \cite{lurie2004derived,toen2008homotopical}.  We will apply the spectral algebraic geometry version from \cite{lu-SAG}.  Specifically, we define and study relative effective Cartier divisors in this context.  By imposing certain conditions, we further define derived level structures on relevant geometric objects in spectral algebraic geometry.  Using the spectral Artin representability theorem, we prove representability results for moduli problems that arise from our derived level structures.  Moreover, we give some examples of applications involving derived level structures.  In particular, we consider the moduli problem of spectral deformations with derived level structures on $p$-divisible groups.  We prove that these moduli problems are representable by certain formal affine spectral Deligne--Mumford stacks, and the corresponding $\einf$-ring spectra provide us with interesting generalized cohomology theories.  

\begin{remark}
    We note that the Goerss--Hopkins--Miller--Lurie sheaf of $\einf$-ring spectra does not directly apply to the moduli problems we consider here, due to the latter's failure of \'etaleness over the base moduli problems without level structure (cf.~the moduli problem in \cite{Devalapurkar} and see also Remarks \ref{rmk:nonlocal} and \ref{rmk:LT} below).  This is fixed by relative effective Cartier divisors analogous to Drinfeld's original approach to arithmetic moduli of (classical) elliptic curves (see \cite[Introduction]{katz1985arithmetic}).  
\end{remark}

\subsection*{Notation and terminology}
\begin{itemize}
    \item Let $\CAlg$ denote the $\infty$-category of $\einf$-rings and $\CAlg^{\cn}$ denote the $\infty$-category of connective $\einf$-rings.  

    \item Let $\cs$ denote the $\infty$-category of spaces ($\infty$-groupoids).  

    \item Given a spectral Deligne--Mumford stack $\sx=(\cx,\so_{\cx})$, let $\tau_{\leq n}\sx$ denote its $n$-truncation $(\cx,\tau_{\leq n} \so_{\cx})$ and $\sx^{\heartsuit}$ denote its underlying ordinary stack $(\cx^{\heartsuit},\tau_{\leq0}\so_{\cx})$.  

    \item By a spectral Deligne--Mumford stack $\sx$ over an $\einf$-ring $R$, we mean a morphism of spectral Deligne--Mumford stacks $\sx\to\Spet R$.  Given an $R$-algebra $S$, we sometimes write $\sx\times_R S$ for the fiber product $\sx \times_{\Spet R} \Spet S$.  

    \item Let $\cm_\El$ denote the spectral Deligne--Mumford stack of spectral elliptic curves, as defined in \cite{lu-EC1}, and $\cm^\cl_\El$ denote the (classical) Deligne--Mumford stack of (classical) elliptic curves.  
\end{itemize}

\subsection*{Overview and statement of results}

Throughout the paper, we work in the $\infty$-categorical setting of spectral algebraic geometry after Lurie, while referencing and comparing with classical moduli problems as treated in, e.g., \cite{katz1985arithmetic,kollar2013rational}.  

In Section \ref{sec:divisor}, we define derived isogenies between spectral elliptic curves and prove that the kernel of a derived isogeny in some cases has properties as in the classical case.  This provides evidence that our derived versions of level structures must induce classical level structures.  For representability reasons, we use moduli associated with sheaves to detect higher homotopy of derived versions of level structures.  

As a main tool, we define relative effective Cartier divisors.  Given a spectral Deligne--Mumford stack $\sx$ over a spectral Deligne--Mumford stack $\ss$, a relative effective Cartier divisor is a morphism $\sd\to\sx$ of spectral Deligne--Mumford stacks such that it is a closed immersion, the ideal sheaf of $\sd$ is a line bundle over $\sx$, and the morphism is flat, proper, and locally almost of finite presentation (Definition \ref{def:cdiv}).  Denote by $\CDiv(\sx/\ss)$ the space of such divisors.  We use Lurie's representability theorem to prove that relative effective Cartier divisors are representable in certain cases.  A major part of our proof involves computing with the associated cotangent complex.  

\begin{thm}[Theorem \ref{thm:cdiv}]
    Suppose that $\se$ is a spectral algebraic space over a connective $\einf$-ring $R$, such that $\se\to R$ is flat, proper, locally almost of finite presentation, geometrically reduced, and geometrically connected.  Then the functor 
    \[
        \CDiv_{\se/R}\co\CAlg^\cn_R\to\cs,\quad R'\mapsto\CDiv(\se_{R'}/R') 
    \]
    is representable by a spectral algebraic space which is locally almost of finite presentation over $R$.  
\end{thm}

In Section \ref{sec:level}, we first define derived level structures for spectral elliptic curves.  Roughly speaking, given an abstract finite abelian group $A$, e.g., $\bz/N\bz$ and $\bz/N\bz\times\bz/N\bz$, a derived level-$A$ structure on a spectral elliptic curve $\se$ over an $\einf$-ring $R$ is just a relative effective Cartier divisor $\sd\to\se$ whose restriction to the heart comes from a classical level-$A$ structure (see Remark \ref{rmk:derived} for more details concerning the {\em derived} nature of such structures).  We let $\Level(A,\se/R)$ denote the space of derived level-$A$ structures on $\se/R$ and show that moduli problems associated with such structures are representable relative to $\se/R$.  

\begin{thm}[Theorem \ref{thm:elevel}]
    Suppose that $\se$ is a spectral elliptic curve over a connective $\einf$-ring $R$.  Then the functor 
    \[
        \Level^A_{\se/R}\co\CAlg^\cn_R\to\cs,\quad R'\mapsto\Level(A,\se_{R'}/R') 
    \]
    is representable by an affine spectral Deligne--Mumford stack which is locally almost of finite presentation over $R$.  
\end{thm}

In classical algebraic geometry, closely related to elliptic curves and more general abelian varieties, we are interested in level structures on $p$-divisible groups.  They can be defined through full sets of sections of commutative finite flat group schemes.  In Section \ref{subsec:plevel}, we consider derived level structures on spectral $p$-divisible groups.  Let $\Level(r,\bG/R)$ denote the space of derived level-$(\bz/p^r\bz)^h$ structures on a height-$h$ spectral $p$-divisible group $\bG/R$.  Again, this moduli problem is relatively representable.  

\begin{thm}[Theorem \ref{thm:plevel}]
    Suppose that $\bG$ is a spectral $p$-divisible group of height $h$ over a connective $\einf$-ring $R$.  Then the functor 
    \[
        \Level^r_{\bG/R}\co\CAlg^\cn_R\to\cs,\quad R'\to\Level(r,\bG_{R'}/R') 
    \]
    is representable by an affine spectral Deligne--Mumford stack $\Spet\cp^r_{\bG/R}$.  
\end{thm}

The remaining Sections \ref{sec:moduli} and \ref{sec:dual} give applications of derived level structures of relevance to algebraic geometry, algebraic topology, and number theory.  We first prove in Section \ref{subsec:mella} that the moduli problem of spectral elliptic curves with derived level-$A$ structure is (absolutely) representable by a spectral Deligne--Mumford stack, based on and generalizing \cite[Theorem 2.4.1]{lu-EC1} (cf.~Remark \ref{rmk:tlevel}).  

\begin{thm}[Theorem \ref{thm:mella}]
    Let $\El^A(R)$ denote the $\infty$-category of spectral elliptic curves with derived level-$A$ structure over a connective $\einf$-ring $R$, with $\El^A(R)^\simeq$ the largest Kan complex contained in $\El^A(R)$.  Then the functor 
    \[
        \cm^A_\El\co\CAlg^\cn\to\cs,\quad R\mapsto\El^A(R)^\simeq 
    \]
    is representable by a spectral Deligne--Mumford stack, and this stack is locally almost of finite presentation over the sphere spectrum $\bs$.  
\end{thm}

In \cite{lu-EC2}, Lurie considers spectral deformations of classical $p$-divisible groups.  As we develop the concept of derived level structures, it is natural to consider the moduli of spectral deformations with derived level structures for certain $p$-divisible groups.  Suppose $\bG_0$ is a $p$-divisible group of height $h$ over a perfect $\bF_p$-algebra $R_0$.  For each nonnegative integer $r$, we consider a functor 
\[
    \Def_{\bG_0}^{\Or,r}\co\CAlg^\ad_\cpl\to\cs,\quad R\mapsto\Def^\Or(r,\bG_0,R) 
\]
where $\Def^\Or(r,\bG_0,R)$ is the $\infty$-category spanned by quadruples $(\bG,\alpha,e,\lambda)$ with 
\begin{itemize}
    \item $\bG$ a spectral $p$-divisible group over $R$, 
    
    \item $\alpha$ an equivalence class of $\bG_0$-taggings of $\bG$, 
    
    \item $e$ an orientation of the identity component of $\bG$, and 
    
    \item $\lambda\in\Level(r,\bG/R)$ a derived level-$(\bz/p^r\bz)^h$ structure on $\bG$ as above.  
\end{itemize}
Our next main result is the following (absolute) representability.  

\begin{thm}[Theorem \ref{thm:jlr}]
    Let $R^\Or_{\bG_0}$ be the oriented deformation ring of $\bG_0/R_0$ from \cite{lu-EC2}.  Then the functor $\Def_{\bG_0}^{\Or,r}$ above is corepresentable by an $\einf$-ring $\jl_r$, where $\jl_r$ is an $R^\Or_{\bG_0}$-algebra such that $\pi_0\jl_r$ is finite over $\pi_0R^\Or_{\bG_0}$.  
\end{thm}

As our notation indicates, the significance of these {\em Jacquet--Langlands spectra} $\jl_r$ is that they serve as a higher-homotopical realization of the Lubin--Tate tower for $\bG_0/R_0$ (see Remark \ref{rmk:LT}).  

We will give another example in Theorem \ref{thm:ehr} of $\einf$-ring spectra obtained by considering moduli of spectral deformations with (non-full) level structure of a $p$-power order subgroup.  They can be viewed as topological realizations for Strickland's deformation rings of Frobenius.  We note some features here in Remarks \ref{rmk:nonlocal} and \ref{rmk:pile}.  

In Section \ref{sec:dual}, we construct for each classical $p$-divisible group an $\einf$-ring spectrum $\jl$, a Jacquet--Langlands spectrum at infinite level (Definition \ref{def:JL} and Proposition \ref{prop:JL}).  By taking homotopy fixed points, we get a {\em Jacquet--Langlands dual} of Morava E-theory (Definition \ref{def:dual} and Theorem \ref{thm:dual}).  

As an illustration, given a formal group $\hG_0$ of height $h$ over a perfect field $k$ of characteristic $p$, there is a diagram of moduli spaces from arithmetic algebraic geometry: 
\[
    \xymatrix{
        & \cx \ar[ld]_{\GL_h(\cO_K)} \ar[rd]^{\cO_D^\times} & \\
	\LT_{\!K} & & \ch 
    }  
\]
where $K$ is the maximal unramified extension of the $p$-adic completion of $k$, $\LT_{\!K}$ is the Lubin--Tate moduli space of deformations of $\hG_0$ over complete local rings with residue field containing $k$, $\cx$ is the moduli space of deformations at infinite level, $D$ is the central division algebra over $K$ of invariant $1/h$, and $\ch$ is the Drinfeld upper half-space of dimension $h-1$.  The above diagram naturally lifts to the following diagram of $\einf$-ring spectra (or rather, spectral moduli spaces, as we keep the directions of the arrows): 
\[
    \hskip.4cm
    \xymatrix{
        & \jl \ar[ld]_{\GL_h(\bz_p)} \ar[rd]^{\bg_h} & \\
	E_h & & \ll E_h 
    }  
\]
where $E_h$ is the Morava E-theory spectrum associated to $\hG_0/k$, $\bg_h$ is a Morava stabilizer group, and $\ll E_h$ the Jacquet--Langlands dual of $E_h$ from Definition \ref{def:dual}.  

Given this higher-homotopical realization of moduli spaces, we discuss its consequences and implications in Proposition \ref{prop:ss} (a square of homotopy fixed point spectral sequences) and Remark \ref{rmk:ss}, as well as in Section \ref{subsec:further}, where we list questions for further investigation.

\section{Effective Cartier divisors of spectral Deligne--Mumford stacks}
\label{sec:divisor}

A main construct of this paper concerns derived level structures.  We begin with a derived version of isogenies and prove that, in certain cases, the kernel of a derived isogeny behaves similarly as in the classical setting.  This gives evidence that our derived version of level structures must induce classical level structures.  In Section \ref{subsec:cdiv}, we define relative effective Cartier divisors in the setting of spectral algebraic geometry.  We then use Lurie's representability theorem to prove that certain functors associated with relative effective Cartier divisors are representable by spectral Deligne--Mumford stacks.  This paves the way for Section \ref{sec:level}, where we establish specifically the representability of derived level structures for spectral elliptic curves and spectral $p$-divisible groups.  

\subsection{Isogenies of spectral elliptic curves}
\label{subsec:isog}

To define derived level structures, the first question we must address is what higher-categorical analogues of finite abelian groups are.  Let us recall from \cite[Section 7.2.4]{lu-HA} and \cite[Section 2.7]{lu-SAG} some finiteness conditions in the context of $\einf$-rings.  

Let $A$ be an $\einf$-ring and $M$ be an $A$-module.  We say that $M$ is 
\begin{itemize}
    \item {\em perfect}, if it is a compact object of the $\infty$-category $\LMod_A$ of left $A$-modules; 

    \item {\em almost perfect}, if there exists an integer $k$ such that $M\in (\LMod_A)_{\geq k}$ and $M$ is an almost compact object of $ (\LMod_A)_{\geq k}$, that is, $\tau_{\leq n}M$ is a compact object of $\tau_{\leq n}\big((\LMod_A)_{\geq k}\big)$ for all $n\geq0$; 

    \item {\em perfect to order $n$}, if given any filtered diagram $\{N_\alpha\}$ in $(\LMod_A)_{\leq 0}$, the canonical map 
    \[
        \varinjlim_\alpha\Ext^{i}_A(M,N_\alpha)\to\Ext^i_A(M,\varinjlim_\alpha N_\alpha) 
    \]
    is injective for $i=n$ and bijective for $i<n$; 

    \item {\em finitely $n$-presented}, if $M$ is $n$-truncated and perfect to order $n+1$.  
\end{itemize}

Next we recall finiteness conditions on algebras.  We say that a morphism $\phi\co A\to B$ of connective $\einf$-rings is 
\begin{itemize}
    \item {\em of finite presentation}, if $B$ belongs to the smallest full subcategory of $\CAlg_A$ which contains $\CAlg_A^{\rm free}$ and is stable under finite colimits; 
    
    \item {\em locally of finite presentation}, if $B$ is a compact object of $\CAlg_A$; 

    \item {\em almost of finite presentation}, if $B$ is an almost compact object of $\CAlg_A$; 

    \item {\em of finite generation to order $n$}, if the following condition holds: 
    \begin{quote}
        Let $\{C_\alpha\}$ be a filtered diagram of connective $\einf$-rings over $A$ having colimit $C$.  Assume that each $C_\alpha$ is $n$-truncated and that each of the transition maps $\pi_nC_\alpha\to\pi_nC_\beta$ is a monomorphism.  Then the canonical map 
        \[
            \varinjlim_\alpha\Map_{\CAlg_A}(B,C_\alpha)\to\Map_{\CAlg_A}(B,C) 
        \]
        is a homotopy equivalence.  
    \end{quote}

    \item {\em of finite type}, if it is of finite generation to order 0.
\end{itemize}

\begin{proposition}[{\cite[Propositions 2.7.2.1 and 4.1.1.3]{lu-SAG}}]
    Let $\phi\co A\to B$ be a morphism of connective $\einf$-rings.  Then the following conditions are equivalent: 
    \begin{itemize}
        \item The morphism $\phi$ is perfect to order $0$ (resp.~of finite type).  

        \item The commutative ring $\pi_0B$ is finite (resp.~of finite type) over $\pi_0A$.  
    \end{itemize}
\end{proposition}

\begin{definition}[{cf.~\cite[Definition 4.2.0.1]{lu-SAG}}]
    Let $f\co\sx\to\sy$ be a morphism of spectral Deligne--Mumford Stacks.  We say that $f$ is {\em locally of finite type} (resp.~{\em locally of finite generation to order $n$}, {\em locally almost of finite presentation}, {\em locally of finite presentation}) if the following condition holds.  Given any commutative diagram 
    \[
        \xymatrix{
		\Spet B \ar[r] \ar[d] & \sx \ar[d]^f \\
		\Spet A \ar[r]        & \sy 
	}
    \]
    where the horizontal morphisms are \'etale, the $\einf$-ring $B$ is of finite type (resp.~of finite generation to order $n$, almost of finite presentation, locally of finite presentation) over $A$.  
\end{definition}

\begin{definition}[{\cite[Definition 5.2.0.1]{lu-SAG}}]
    Let $f\co(\cx,\so_\cx)\to(\cy,\so_\cy)$ be a morphism of spectral Deligne--Mumford stacks.  We say that $f$ is {\em finite} if the following conditions hold: 
    \begin{enumerate}
	\item The morphism $f$ is affine, and 
 
	\item The pushforward $f_*\so_\cx$ is perfect to order 0 as a $\so_\cy$-module.  
    \end{enumerate}
\end{definition}

\begin{remark}
\label{rmk:finite}
    By \cite[Example 4.2.0.2]{lu-SAG}, a morphism $f\co\sx\to\sy$ of spectral Deligne--Mumford stack is locally of finite type if and only if the underlying map of ordinary stacks is locally of finite type in the sense of classical algebraic geometry.  Moreover, by \cite[Remark 5.2.0.2]{lu-SAG}, a morphism of $f\co\sx\to\sy$ is finite if and only if the underlying map $f^\heartsuit\co\sx^\heartsuit\to\sy^\heartsuit$ is finite.  In particular, if $\sx$ and $\sy$ are spectral algebraic spaces, then $f$ is finite if and only if $f^\heartsuit$ is finite in the classical sense.  
\end{remark}

Recall that a morphism $f\co\sx\to\sy$ of spectral Deligne--Mumford stacks is {\em surjective}, if for every field $k$ and any map $\Spet k\to\sy$, the fiber product $\Spet k\times_\sy\sx$ is nonempty \cite[Definition 3.5.5.5]{lu-SAG}.  

\begin{definition}
\label{def:isog}
    Let $R$ be a connective $\einf$-ring and $f\co\sx\to\sy$ be a morphism of spectral abelian varieties over $R$.  We call $f$ an {\em isogeny} if it is finite, flat, and surjective.  
\end{definition}

\begin{lemma} 
    Let $f\co\sx\to\sy$ be an isogeny of spectral abelian varieties.  Then $f^\heartsuit\co\sx^\heartsuit\to\sy^\heartsuit$ is an isogeny in the classical sense.  
\end{lemma}
\begin{proof}
    For ordinary abelian varieties, $f^\heartsuit$ being an isogeny means that it is surjective and its kernel is finite.  This is equivalent  to $f^{\heartsuit}$ being finite, flat, and surjective \cite[Proposition 7.1]{milne1986abelian}.  From Definition \ref{def:isog}, it is clear that $f^{\heartsuit}$ is finite and flat.  We need only show that $f^\heartsuit$ is surjective.  
    
    By definition of surjectivity above for morphisms of spectral Deligne--Mumford stacks, we get a commutative diagram
    \[
        \xymatrix{
		\Spet k' \ar[d] \ar[r] & \sx \ar[d] \\
		\Spet k \ar[r]         & \sy 
	}
    \]
    The upper horizontal morphism corresponds to a morphism $\Spet k'\to\sx^\heartsuit$ by the inclusion--truncation adjunction \cite[Proposition 1.4.6.3]{lu-SAG}.  On underlying topological spaces, this then corresponds to a point $|\Spet k'|\to|\sx^\heartsuit|$.  It is clear that this point in $|\sx^\heartsuit|$ is a preimage of $|\Spet k|$ in $|\sy^\heartsuit|$.  Therefore $f^\heartsuit$ is surjective.  
\end{proof}

\begin{lemma}
\label{lem:fib}
    Let $f\co\sx\to\sy$ be an isogeny of spectral elliptic curves over a connective $\einf$-ring $R$.  Then $\fib(f)$ exists and is a finite and flat nonconnective spectral Deligne--Mumford stack over $R$.  
\end{lemma}
\begin{proof}
    By \cite[Proposition 1.4.11.1]{lu-SAG}, finite limits of nonconnective spectral Deligne--Mumford stacks exist, so we can define $\fib(f)$.  Let us consider the commutative diagram 
    \[
        \xymatrix{
		\fib(f) \ar[r] \ar[d]^{f'} & \sx \ar[d]^f \ar[rdd] & \\
		\!\,\ast \ar[r] \ar[rrd]^i & \sy \ar[rd]           & \\
                                       &                       & \Spet R 
	}
    \]
    where the square is a pullback diagram.  We find that $\fib(f)$ is over $\Spet R$. By \cite[Remark 2.8.2.6]{lu-SAG}, $f'\co\fib(f)\to\ast$ is flat because it is a pullback of a flat morphism.  Clearly $i\co\ast\to\Spet R$ is flat, so by \cite[Example 2.8.3.12]{lu-SAG} (being a flat morphism is a property local on the source with respect to the flat topology), $i\circ f'\co\fib(f)\to\Spet R$ is flat.  

    Next we show that $\fib(f)$ is finite over $R$.  Since $\ast$, $\sx$, and $\sy$ are all spectral algebraic spaces, so is $\fib(f)$.  Moreover, $\Spet R$ is a spectral algebraic space \cite[Example 1.6.8.2]{lu-SAG}.  By Remark \ref{rmk:finite}, we need only prove that the underlying morphism is finite.  Since the truncation functor is a right adjoint, it preserves limits.  Thus we get a pullback diagram 
    \[
        \xymatrix{
		\fib(f)^\heartsuit \ar[r] \ar[d] & \sx^\heartsuit \ar[d] \\
		\!\,\ast \ar[r]                  & \sy^\heartsuit 
	}
    \]
    So we are reduced to showing that given an isogeny $f^\heartsuit\co\sx^\heartsuit\to\sy^\heartsuit$ of ordinary abelian varieties over a commutative ring $R$, its kernel is finite over $R$.  This is true in classical algebraic geometry \cite[Proposition 7.1]{milne1986abelian}.  
\end{proof}

\begin{lemma}
    Given an integer $N\geq1$, let $f_N\co\se\to\se$ be an isogeny of spectral elliptic curves over a connective $\einf$-ring $R$ such that the underlying morphism is the multiplication-by-$N$ map $[N]\co\se^\heartsuit\to\se^\heartsuit$.  Then $\fib(f_N)$ is finite flat of degree $N^2$ in the sense of \cite[Definition 5.2.3.1]{lu-SAG}.  Moreover, if $N$ is invertible in $\pi_0R$, then $\fib(f_N)$ is an \'etale-locally constant sheaf.  
\end{lemma}
\begin{proof}
    By \cite[Theorem 2.3.1]{katz1985arithmetic}, we know that $[N]\co\se^\heartsuit\to\se^\heartsuit$ is finite locally free of rank $N^2$ in the classical sense.  When $N$ is invertible in $\pi_0R$, its kernel is an \'etale-locally constant sheaf.  Now, from Lemma \ref{lem:fib}, $\fib(f_N)$ is a spectral algebraic space that is finite and flat, and its underlying space $\fib(f_N)^\heartsuit= \ker[N]$ is locally free of rank $N^2$.  We need to prove that $\fib(f_N)\to\Spet R$ is locally free of rank $N^2$ in spectral algebraic geometry.  Observe that since $\fib(f_N)$ is finite and flat, it is affine.  We are thus reduced to proving the above for affines, i.e., $f_N|_{\Spet S}\co\Spet S\to\Spet R$ is locally free of rank $N^2$ for any affine substack $\Spet S$ of $\fib(f_N)$.  This is equivalent to proving that $R\to S$ is locally free of rank $N^2$ in the sense of \cite[Definition 2.9.2.1]{lu-SAG}.  Therefore we need to prove the following: 
    \begin{enumerate}
        \item The ring $S$ is locally free of finite rank over $R$ (by \cite[Proposition 7.2.4.20]{lu-HA}, this is equivalent to saying that $S$ is a flat and almost perfect $R$-module).  

        \item For every $\einf$-ring maps $R\to k$ with $k$ a field, the vector space $\pi_0(k\otimes_RS)$ is an $N^2$-dimensional $k$-vector space.  
    \end{enumerate}

    For (1), we know that $\pi_0S$ is a projective $\pi_0R$-module and that $S$ is a flat $R$-module, so by \cite[Proposition 7.2.2.18]{lu-HA}, $S$ is a projective $R$-module.  By \cite[Corollary 7.2.2.9]{lu-HA}, since $\pi_0S$ is a finitely generated $\pi_0R$-module, $S$ is a retract of a finitely generated free $R$-module, and is therefore locally free of finite rank.  
    
    For (2), by \cite[Corollary 7.2.1.23]{lu-HA}, since $R$ and $S$ are connective, we have $\pi_0(k\otimes_RS)\simeq k\otimes_{\pi_0R}\pi_0S$, which is an $N^2$-dimensional $k$-vector space, as $\pi_0S$ is a rank-$N^2$ free $\pi_0R$-module from above.  
    
    We next show that if $N$ is invertible in $\pi_0R$, then $\fib(f_N)$ is a locally constant sheaf.  Since $\fib(f_N)$ is a spectral Deligne--Mumford stack, its associated functor of points $\fib(f_N)\co\CAlg_R\to\cs$ is nilcomplete and locally almost of finite presentation.  By \cite[Theorem 2.3.1]{katz1985arithmetic}, $\fib(f_N)|_{\CAlg_{\pi_0R}^\heartsuit}$ is a locally constant sheaf.  The desired result then follows from the lemma below.  
\end{proof}

\begin{lemma}
    Let $R$ be a connective $\einf$-ring.  Let $\sf\in\Shv^\et(\CAlg^{\cn}_R)$ be nilcomplete and locally almost of finite presentation.  Suppose that $\sf|_{(\CAlg^\cn_R)^\heartsuit}$ is a locally constant presheaf.  Then $\sf$ is a (homotopy) locally constant sheaf (i.e., sheafification of a homotopy-locally constant presheaf).  
\end{lemma}	
\begin{proof}
    Let us choose an \'etale cover $\{U^0_{i}\}$ of $\pi_0R$ such that $\sf|_{U^0_{i}}$ is a constant sheaf for each $i$.  By \cite[Theorem 7.5.1.11]{lu-HA}, this corresponds to an \'etale cover $\{U_i\}$ of $R$ such that $\pi_0U_i=U^0_i$.  For each $i$ and $n$, we consider the diagram 
    \[
        \xymatrix{
		\tau_{\leq0}R \ar[r] \ar[d] & \tau_{\leq0}U_i \ar[d] \\
		\tau_{\leq n}R \ar[r]       & \tau_{\leq n}U_i 
	}
    \]
    which is a pushout diagram, since $U_i$ is an \'etale $R$-algebra.  This is a colimit diagram in $\tau_{\leq n}\CAlg_R$.  Since $\sf$ is a sheaf locally almost of finite presentation, we then get a pushout diagram
    \[
        \xymatrix{
		\sf(\tau_{\leq0}R) \ar[r] \ar[d] & \sf(\tau_{\leq0}U_i \ar[d]) \\
		\sf(\tau_{\leq n}R) \ar[r]       & \sf(\tau_{\leq n}U_i) 
	}
    \]
    Without loss of generality, we may assume that each $U_i$ is connective. Thus the values $\sf(\tau_{\leq 0}U_i)$ is independent of $i$.  This implies that $\sf(\tau_{\leq n}U_i)$ are all equivalent.  Since $\sf$ is nilcomplete, we have $\sf(U_i)\simeq\varinjlim_n\sf(\tau_{\leq n}U_i)$, and so all $\sf(U_i)$ are equivalent.  
\end{proof}

\subsection{Cartier divisors and an exercise of spectral Artin representability}
\label{subsec:cdiv}

In this subsection, we define relative effective Cartier divisors in the context of spectral algebraic geometry.  We then use Lurie's spectral Artin representability theorem to prove that relative effective Cartier divisors are representable in certain cases.  Let us first recall this spectral analogue of Artin's representability criterion in classical algebraic geometry \cite[Theorem 18.3.0.1]{lu-SAG}.  

\begin{theorem}[Lurie]
\label{thm:Lurie}
    Let $X\co\CAlg^\cn\to\cs$ be a functor.  Suppose that we have a natural transformation $f\co X\to\Spec R$, where $R$ is a Noetherian $\einf$-ring with $\pi_0R$ a Grothendieck ring.  Given $n\geq0$, $X$ is representable by a spectral Deligne--Mumford $n$-stack which is locally almost of finite presentation over $R$ if and only if the following conditions are satisfied: 
    \begin{enumerate}
        \item For every discrete commutative ring $A$, the space $X(A)$ is $n$-truncated.  

        \item The functor $X$ is a sheaf for the \'etale topology.  

        \item The functor $X$ is nilcomplete, infinitesimally cohesive, and integrable.  

        \item The functor $X$ admits a connective cotangent complex $L_X$.  

        \item The natural transformation f is locally almost of finite presentation.  
    \end{enumerate}
\end{theorem}

Given a locally spectrally ringed topos $\sx=(\cx,\so_\cx)$, we can consider its functor of points 
\[
    h_\sx\co\ringtop\to\cs,\quad\sy\mapsto\Map_\ringtop(\sy,\sx) 
\]
In particular, by \cite[Remark 3.1.1.2]{lu-SAG}, a closed immersion $f\co(\cy,\so_\cy)\to(\cx,\so_\cx)$ of locally spectrally ringed topoi corresponds to a morphism $\so_\cx\to f_*\so_\cy$ of sheaves over $\cx$ of connective $\einf$-rings such that $\pi_0\so_\cx\to\pi_0f_*\so_\cy$ is an epimorphism.  We denote this epimorphism by $\alpha$.  Given a closed immersion $f\co\sd\to\sx$ of spectral Deligne--Mumford stacks, we let $\ci(\sd)$ denote $\ker(\alpha)$, called the ideal sheaf of $\sd$.  

To prove relative representability for effective Cartier divisors below, we need the representability of Picard functors.  Given a map $f\co\sx\to\Spet R$ of spectral Deligne--Mumford stacks, we can define a functor 
\[
    \spic_{\sx/R}\co\CAlg^\cn_R\to\cs,\quad R'\mapsto\spic(\Spet R'\times_{\Spet R}\sx) 
\]
If $f$ admits a section $x\co\Spet R\to\sx$, then pullback along $x$ gives a natural transformation of functors $\spic_{\sx/R}\to\spic_{R/R}$.  We let $\spic^x_{\sx/R}\co\CAlg^\cn_R\to\cs$ denote the fiber of this map.  

\begin{theorem}[{\cite[Theorem 19.2.0.5]{lu-SAG}}]
    Let $f\co\sx\to\Spet R$ be a map of spectral algebraic spaces which is flat, proper, locally almost of finite presentation, geometrically reduced, and geometrically connected over an $\einf$-ring $R$.  Suppose that $x\co\Spet R \to\sx$ is a section of $f$.  Then the functor $\spic^x_{\sx/R}$ is representable by a spectral algebraic space which is locally of finite presentation over $R$.  
\end{theorem}

In the classical setting, schemes representing relative effective Cartier divisors are open subschemes of Hilbert schemes \cite[Theorem 1.13]{kollar2013rational}.  However, in the derived setting, the Hilbert functor is representable by a spectral algebraic space \cite[Theorem 8.3.3]{lurie2004derived}, and it is hard to establish an analogous relationship.  We will directly study relative effective Cartier divisors and their spectral moduli as follows.  

\begin{definition}[Relative effective Cartier divisor]
\label{def:cdiv}
    Let $\sx$ be a spectral Deligne--Mumford stack over a spectral Deligne--Mumford stack $\ss$.  Define a {\em relative effective Cartier divisor of $\sx/\ss$} to be a closed immersion $\sd\to\sx$ such that it is flat, proper, locally almost of finite presentation and that the associated ideal sheaf of $\sd$ over $\sx$ is locally free of rank 1.  We let $\CDiv(\sx/\ss)$ denote the $\infty$-category of such closed immersions.  
\end{definition}

\begin{remark}
\label{rmk:Kan}
    It is not hard to see that given any spectral Deligne--Mumford stack $\sx$ over $\ss$, $\CDiv(\sx/\ss)$ is a Kan complex, since all objects are closed immersions of $\sx$.  Let $\sd\to\sd'$ be a morphism.  Then we have a diagram 
    \[
        \xymatrix{
    	\sd \ar[rr]^f \ar[rd] &     & \sd' \ar[ld] \\
                                  & \sx & 
        }
    \]
    By the definition of closed immersions, they are all equivalent to the same substack of $\sx$, so $f$ is an isomorphism (cf.~\cite[Remark 3.1.1.2]{lu-SAG}).  
\end{remark}

\begin{lemma}
\label{lem:bc}
    Let $\sx/\ss$ be a spectral Deligne--Mumford stack as above, and $\st\to\ss$ be a map of spectral Deligne--Mumford stacks.  If we have a relative effective Cartier divisor $\sd\to\sx$, then $\sd_\st$ is a relative effective Cartier divisor of $\sx_\st$.  
\end{lemma}
\begin{proof}
    This is straightforward to check.  We simply note that $\sd_\st$ is a closed immersion of $\sx_\st$ \cite[Corollary 3.1.2.3]{lu-SAG}.  After base change, $\sd_\st$ is flat, proper, and locally almost of finite presentation over $\st$.  It remains to show that $\ci(\sd_\st)$ is a line bundle over $\sx_\st$.  Indeed, we have a fiber sequence 
    \[
        \ci(\sd)\to\so_\cx\to\so_\cd 
    \]
    By the flatness of $\sd$, pullback along the base change $f\co\st\to\ss$ gives another fiber sequence 
    \[
        f^*\big(\ci(\sd)\big)\to\so_{\cx_\st}\to\so_{\cd_\st} 
    \]
    So we have that $\ci(\sd_\st)$ is just $f^*\big(\ci(\sd)\big)$, which is invertible.  
\end{proof}

Suppose that $\sx$ is a spectral Deligne--Mumford stack over an affine spectral Deligne--Mumford stack $\ss=\Spet R$. From Definition \ref{def:cdiv}, we then have a functor 
\[
    \CDiv_{\sx/R}\co\CAlg^\cn_R\to\cs,\quad R'\mapsto\CDiv(\sx_{R'}/R') 
\]
Our main goal in this section is to prove that this functor is representable when $\sx/R$ is a spectral algebraic space satisfying certain conditions.  To achieve this, we need some preparations for computing the cotangent complex of a relative effective Cartier divisor functor.  The main issue has to do with square-zero extensions, for which we need the following facts about pushouts of two closed immersions.  

By \cite[Theorem 16.2.0.1 and Proposition 16.2.3.1]{lu-SAG}, given a pushout square of spectral Deligne--Mumford stacks 
\[
    \xymatrix{
        \sx_{01} \ar[r]^i \ar[d]^j & \sx_0 \ar[d]^{j'} \\
        \sx_1 \ar[r]^-{i'}         & \sx 
    }
\]
such that $i$ and $j$ are closed immersions, the induced square of $\infty$-categories 
\[
    \xymatrix{
        \QCoh(\sx_{01})     & \QCoh(\sx_0) \ar[l] \\
        \QCoh(\sx_1) \ar[u] & \QCoh(\sx) \ar[l] \ar[u] 
    }
\]
determines an embedding $\theta\co\QCoh(\sx)\to\QCoh(\sx_0)\times_{\QCoh(\sx_{01})}\QCoh(\sx_1)$, which restricts to an equivalence between connective objects 
\[
    \QCoh(\sx)^\cn\to\QCoh(\sx_0)^\cn\times_{\QCoh(\sx_{01})^\cn}\QCoh(\sx_1)^\cn 
\]
Moreover, let $\sf\in\QCoh(\sx)$, and set $\sf_0=j'^*\sf\in\QCoh(\sx_0)$ and $\sf_1=i'^*\sf\in\QCoh(\sx_1)$.  Then $\sf$ is $n$-connective if and only if $\sf_0$ and $\sf_1$ are $n$-connective, and this statement is also true for the conditions of almost connective, Tor-amplitude $\leq n$, flat, perfect to order $n$, almost perfect, perfect, and locally free of finite rank, respectively.  

Also, by \cite[Theorem 16.3.0.1]{lu-SAG}, we have a pullback square of $\infty$-categories 
\[
    \xymatrix{
        \SpDM_{/\sx} \ar[r] \ar[d] & \SpDM_{/\sx_0} \ar[d] \\
        \SpDM_{/\sx_1} \ar[r]      & \SpDM_{/\sx_{01}} 
    }
\]
Let $f\co\sy\to\sx$ be a map of spectral Deligne--Mumford stacks.  Let $\sy_0=\sx_0\times_\sx\sy$, $\sy_1=\sx_1\times_\sx\sy$, and let $f_0\co\sy_0\to\sx_0$ and $f_1\co\sy_1\to\sx_1$ be the projection maps.  Then we have that $f$ is locally almost of finite presentation if and only if both $f_0$ and $f_1$ are locally almost of finite presentation.  The statement remains true for the following individual conditions: locally of finite generation to order $n$, locally of finite presentation, \'etale, equivalence, open immersion, closed immersion, flat, affine, separated, and proper \cite[Proposition 16.3.2.1]{lu-SAG}.  

Now, let $\sx=(\cx,\so_\cx)$ be a spectral Deligne--Mumford stack, $\sE\in\QCoh(\sx)^\cn$ be a connective quasi-coherent sheaf, and $\eta\in\Der(\so_\cx,\Sigma\sE)$ be a derivation, i.e., a morphism $\eta\co\so_\cx\to\so_\cx\oplus\Sigma\sE$.  We let $\so_\cx^\eta$ denote the square-zero extension of $\so_\cx$ by $\sE$ determined by $\eta$, so that we  have a pullback diagram 
\[
    \xymatrix{
        \so_\cx^\eta \ar[r] \ar[d] & \so_\cx \ar[d]^\eta \\
        \so_\cx \ar[r]^-0          & \so_\cx\oplus\Sigma\sE 
    }
\]
By \cite[Proposition 17.1.3.4]{lu-SAG}, $(\cx,\so_\cx^\eta)$ is a spectral Deligne--Mumford stack, which we will denote by $\sx^\eta$.  In the case of $\eta=0$, we denote it by $\sx^\sE=(\cx,\so_\cx\oplus\sE)$.  We then have a pushout square of spectral Deligne--Mumford stacks
\[
    \xymatrix{
        \sx^\sE    & \sx \ar[l] \\
        \sx \ar[u] & \sx^{\Sigma\sE} \ar[l]^-g \ar[u]^f 
    }
\]
such that $f$ and $g$ are closed immersions.  In turn, by \cite[Theorem 16.2.0.1]{lu-SAG}, there is a pullback diagram 
\[
    \xymatrix{
        \QCoh(\sx^\sE)^\acn \ar[d] \ar[r] & \QCoh(\sx)^\acn \ar[d] \\
        \QCoh(\sx)^\acn \ar[r]            & \QCoh(\sx^{\Sigma\sE})^\acn 
    }
\]
of categories spanned by almost connective quasi-coherent sheaves.  Passing to homotopy fibers over some $\sf\in\QCoh(\sx)^\acn$, we obtain an equivalence 
\[
    \QCoh(\sx^\sE)^\acn\times_{\QCoh(\sx)}\{\sf\}\simeq\Map_{\QCoh(\sx)}\big(\sf,\Sigma(\sE\otimes\sf)\big) 
\]
as in \cite[Proposition 19.2.2.2]{lu-SAG}.  Similarly, by passing to the homotopy fibers over some $\sz\in\SpDM_{/\sx}$ with $f\co\sz\to\sx$, we obtain the classification of first-order deformations of $\sx$: 
\[
    \SpDM_{/\sx^\sE}\times_{\SpDM_{/\sx}}\{\sz\}\simeq\Map_{\QCoh(\sz)}(L_{\sz/\sx},\Sigma f^*\sE) 
\]
\cite[Proposition 19.4.3.1]{lu-SAG}.  

\begin{lemma}
    Let $f\co\sx\to\Spet R$ be a morphism of spectral Deligne--Mumford stacks, and $M$ be a connective $R$-module.  Consider the $\infty$-category of Deligne--Mumford stacks $\sx'$ equipped with a morphism $f'\co\sx'\to\Spet(R\oplus M)$ that fits into the pullback diagram 
    \[
        \xymatrix{
            \sx \ar[r] \ar[d]_f & \sx' \ar[d]^{f'} \\
            \Spet R \ar[r]      & \Spet(R\oplus M) 
        }
    \]
    Then this $\infty$-category is a Kan complex, and it is canonically homotopy equivalent to the mapping space $\Map_{\QCoh(\sx)}(L_{\sx/{\Spet R}},\Sigma f^*M)$.  Moreover, if $f$ is flat, proper, and locally almost of finite presentation, then so is $f'$.  
\end{lemma}
\begin{proof}
    We have a pullback square of $\einf$-rings 
    \[
        \xymatrix{
    	R\oplus M \ar[d] \ar[r] & R \ar[d]^{(\id,0)} \\
    	R \ar[r]                & R\oplus\Sigma M 
        }
    \]
    which corresponds to a pushout square of spectral Deligne--Mumford stacks 
    \[
        \xymatrix{
            \Spet R\oplus M & \Spet R \ar[l] \\
            \Spet R \ar[u]  & \Spet(R\oplus\Sigma M) \ar[u] \ar[l] 
        }
    \]
    such that the morphisms $\Spet(R\oplus\Sigma M)\to\Spet R$ are closed immersions.  This exhibits $\Spet(R\oplus M)$ as an ``infinitesimal thickening'' of $\Spet R$ determined by $R\xrightarrow{(\id,0)}R\oplus\Sigma M$.  
    
    The first part of this lemma follows from the formula for first-order deformations of \cite[Proposition 19.4.3.1]{lu-SAG}.  The second part follows from properties of pushout of two closed immersions \cite[Corollary 16.4.2.1]{lu-SAG}.  
\end{proof}

\begin{lemma}
\label{lem:def}
    Suppose that we are given a pushout diagram of spectral Deligne--Mumford stacks 
    \[
        \xymatrix{
            \sx_{01} \ar[r]^{i}  \ar[d]^{j} & \sx_0 \ar[d] \\
            \sx_1 \ar[r]                    & \sx \\
        }
    \]
    where i and j are closed immersions.  Let $f\co \sy\to\sx$ be a map of spectral Deligne--Mumford stacks.  Let $\sy_0=\sx_0\times_\sx\sy$, $\sy_1=\sx_1\times_\sx\sy$, and let $f_0\co \sy_0\to \sx_0$ and $f_1\co \sy_1\to \sx_1$ be the projection maps.  If $f_0$ and $f_1$ are both closed immersions and determine line bundles over $\sy_0$ and $\sy_1$ respectively, then $f$ is a closed immersion and determines a line bundle over $\sy$.  
\end{lemma}
\begin{proof}
    The statement concerning closed immersions follows from \cite[Proposition 16.3.2.1]{lu-SAG}.  For the line-bundle part, we note that by \cite[Theorem 16.2.0.1 and Proposition 16.2.3.1]{lu-SAG}, $f$ determines a sheaf locally free of finite rank.  To show that this sheaf is a line bundle, we proceed locally.  By \cite[Theorem 16.2.0.2]{lu-SAG}, given a pullback diagram of connective $\einf$-rings 
    \[
        \xymatrix{
            A \ar[r] \ar[d] & A_0 \ar[d] \\
            A_1 \ar[r]      & A_{01} 
        }
    \]
    such that $\pi_0A_0\to\pi_0A_{01}\leftarrow\pi_0A_1$ are surjective, there is an equivalence 
    \[
        F\co\Mod^\cn_A\to\Mod^\cn_{A_0}\times_{\Mod^\cn_{A_{01}}}\Mod^\cn_{A_1} 
    \]
    Moreover, this is a symmetric monoidal equivalence.  Indeed, since 
    \[
        F(M)=(A_0\otimes_AM,A_1\otimes_AM,A_{01}\otimes_{A_0}A_0\otimes_AM\simeq A_{01}\otimes_{A_1}A_1\otimes_AM) 
    \]
    we have $F(M\otimes_AN)\simeq F(M)\otimes F(N)$.  By \cite[Proposition 2.9.4.2]{lu-SAG}, line bundles over $A_1$, $A_{01}$, and $A_0$ determine invertible objects of $\Mod^\cn_{A_1}$, $\Mod^\cn_{A_{01}}$, and $\Mod^\cn_{A_1}$ respectively, which in turn determine an invertible object of $\Mod^\cn_A$, hence a line bundle over $A$.  
\end{proof}

Here is the main result of this section and the technical heart of the paper.  

\begin{theorem}
\label{thm:cdiv}
    Given a connective $\einf$-ring $R$, let $\se/R$ be a spectral algebraic space that is flat, proper, locally almost of finite presentation, geometrically reduced, and geometrically connected.  Then the functor 
    \[
        \CDiv_{\se/R}\co\CAlg^\cn_R\to\cs,\quad R'\mapsto\CDiv(\se_{R'}/R') 
    \]
    is representable by a spectral algebraic space which is locally almost of finite presentation over $\Spet R$.  
\end{theorem}
\begin{proof}
    We apply Lurie's spectral Artin representability theorem and verify the 5 criteria from Theorem \ref{thm:Lurie} one by one, in the case of $n=0$, as follows: 
    \begin{enumerate}
        \item Lemma \ref{lem:1}; 

        \item Lemma \ref{lem:2}; 

        \item Lemmas \ref{lem:3a}, \ref{lem:3b}, \ref{lem:3c}; 

        \item Lemma \ref{lem:4}; and 

        \item Lemma \ref{lem:5}.  
    \end{enumerate}
    These statements and their proofs occupy the rest of this section.  
\end{proof}

\begin{lemma}
\label{lem:1}
    For every discrete commutative $R_0$, the space $\CDiv_{\se/R}(R_0)$ is $0$-truncated.  
\end{lemma}
\begin{proof}
    Recall that $\CDiv_{\se/R}(R_0)$ consists of closed immersions $\sd\to\se\times_RR_0$ such that $\sd$ is flat and proper over $R_0$.  Therefore, if $R_0$ is discrete, so are the objects $\sd$, and so $\CDiv_{\se/R}(R_0)$ is 0-truncated.  
\end{proof}

\begin{lemma}
\label{lem:2}
    The functor $\CDiv_{\se/R}$ is a sheaf for the \'etale topology.  
\end{lemma}
\begin{proof}
    Let $\{R'\to U_i\}_{i\in I}$ be an \'etale cover of $\Spet R'$, and $U_\bullet$ be the associated \v{C}ech-simplicial object.  We need to prove that the map 
    \[
        \CDiv_{\se/R}(R')\to\varprojlim_\bDelta\CDiv_{\se/R}(U_\bullet) 
    \]
    is an equivalence.  Unwinding the definitions, we need only prove the following general result: Given a spectral Deligne--Mumford stack $\sx/\ss$ and an \'etale cover $\st_i\to\ss$, we have a homotopy equivalence 
    \[
        \CDiv(\sx/\ss)\to\varprojlim_\bDelta\CDiv(\sx\times_\ss \st_\bullet) 
    \]
    This follows from the fact that our conditions on relative effective Cartier divisors from Definition \ref{def:cdiv} are local with respect to the \'etale topology.  
\end{proof}

\begin{lemma}
\label{lem:3a}
    The functor $\CDiv_{\se/R}$ is nilcomplete.  
\end{lemma}
\begin{proof}
    By \cite[Definition 17.3.2.1]{lu-SAG}, we need to show that the canonical map 
    \[
        \CDiv_{\se/R}(R')\to\varprojlim_n\CDiv_{\se/R}(\tau_{\leq n}R') 
    \]
    is a homotopy equivalence for every $\einf$-ring $R'$.  This can be deduced from the following: Given a flat, proper, locally almost of finite presentation spectral algebraic space $\sx$ over a connective $\einf$-ring $S$, we have an equivalence 
    \[
        \CDiv(\sx/S)\to\varprojlim_n\CDiv(\sx\times_S\tau_{\leq n}S) 
    \]
        
    Let us now prove this equivalence.  Given a relative effective Cartier divisor $\sd\to\sx$, we have the following commutative diagram 
    \[
        \xymatrix{
            \sd\times_S\tau_{\leq n}S \ar@{-->}[d] \ar@/_3pc/[dd] \ar[r] & \sd \ar[d] \\
		\sx\times_S\tau_{\leq n}S \ar[d] \ar[r]                      & \sx \ar[d] \\
		\Spet\tau_{\leq n}S \ar[r]                                   & \Spet S 
	}
    \]
    where we get an induced map $\sd\times_S\tau_{\leq n}S\to\sx\times_S\tau_{\leq n}S$.  It is not hard to prove that this map is a closed immersion \cite[Corollary 3.1.2.3]{lu-SAG}.  Moreover, the map $\sd\times_S\tau_{\leq n}S\to\Spet\tau_{\leq n}S$ is flat, proper, and locally almost of finite presentation, since $\sd\times_S\tau_{\leq n}S$ is the base change of $\sd$ along $\Spet\tau_{\leq n}S\to\Spet S$.  The associated ideal sheaf of $\sd\times_S\tau_{\leq n}S$ remains a line bundle over $\sx\times_S\tau_{\leq n}S$.  Therefore $\sd\times_S\tau_{\leq n}S$ is a relative effective Cartier divisor of $\sx\times_S\tau_{\leq n}S$.  Thus we define a functor 
    \[
        \theta\co\CDiv(\sx/S)\to\varprojlim_n\CDiv(\sx\times_S\tau_{\leq n}S),\quad\sd\mapsto\{\sd\times_S\tau_{\leq n}S\}_n 
    \]
    This functor is fully faithful, since we have from \cite[Proposition 19.4.1.2]{lu-SAG} an equivalence 
    \[
        \SpDM_{/S}\to\varprojlim_n\SpDM_{/\tau_{\leq n}S} 
    \]
    defined by $\sx\mapsto\sx\times_S\tau_{\leq n}S$.  For $\theta$ to be an equivalence, we need only show that it is essentially surjective.  
        
    Suppose $\{\sd_n\to\sx\times_S\tau_{\leq n}S\}_n$ is an object in $\varprojlim_n\CDiv(\sx\times_S\tau_{\leq n}S)$.  It is a morphism in $\varprojlim_n\SpDM_{/\tau_{\leq n}S}$.  By \cite[Proposition 19.4.1.2]{lu-SAG}, there is a morphism $\sd\to\sx$ in $\SpDM_{/S}$ such that $\sd\times_S\tau_{\leq n}S\to\sx\times_S\tau_{\leq n}S$ are equivalent to $\sd_n\to\sx\times_S\tau_{\leq n}S$.  

    Next, we need to show that $\sd\to\sx$ from above is a relative effective Cartier divisor.  The conditions that $\sd\to\sx$ is flat, proper, and locally almost of finite presentation follow immediately from \cite[Proposition 19.4.2.1]{lu-SAG}.  It remains to prove that $\sd\to\sx$ is a closed immersion and determines a line bundle over $\sx$.  
        
    Without loss of generality, we may assume that $\sx=\Spet B$ is affine, so that we have closed immersions 
    \[
        \sd_n\to(\Spet B)\times_S\tau_{\leq n}S\simeq\Spet(B\otimes_S\tau_{\leq n}S) 
    \]
    the last equivalence from \cite[Proposition 1.4.11.1(3)]{lu-SAG}.  By \cite[Theorem 3.1.2.1]{lu-SAG}, each $\sd\times_S\tau_{\leq n }S$ is equivalent to $\Spet B'_n$ for some $B'_n$ such that $\pi_0(B\otimes_S\tau_{\leq n}S)\to\pi_0B_n'$ is surjective.  Since $\tau_{\leq n+1}S\to B'_{n+1}$ is flat, we have 
    \[
        \begin{split}
            \Spet B'_n & =(\Spet B'_{n+1})\times_{\tau_{\leq n+1}S}\tau_{\leq n}S=\Spet(B'_{n+1}\otimes_{\tau_{\leq n+1}S}\tau_{\leq n}S) \\
                       & \simeq\Spet\tau_{\leq n}B'_{n+1} 
        \end{split}
    \]
    Thus we obtain a spectrum $B'$ such that 
    \[
        \Spet\tau_{\leq n}B'\simeq\Spet B'_n=\sd\times_S\tau_{\leq n} 
    \]
    Consequently, $\sd=\Spet B'$ and $\pi_0B\to\pi_0B'$ is surjective, and so 
    \[
        \sd=\Spet B'\to\Spet B=\sx 
    \]
    is a closed immersion.  

    Finally, to prove that the associated ideal sheaf of $\sd$ is a line bundle, we note the pullback diagrams 
    \[
        \xymatrix{
		I_n \ar[r] \ar[d] & B\otimes_S\tau_{\leq n}S \ar[d] \\
            \ast \ar[r]       & B'\otimes_S\tau_{\leq n}S 
	}
    \]
    where each $I_n$ is an invertible module over $B\otimes_S\tau_{\leq n}S=\tau_{\leq n}B$.  Passing to inverse limits, we obtain a pullback diagram 
    \[
        \hskip-.85cm
        \xymatrix{
		\varprojlim I_n \ar[r] \ar[d] & B \ar[d] \\
		\ast \ar[r]                   & B' 
	}
    \]
    Consequently, we have $I(\sd)\simeq\varprojlim I_n$.  Now, by nilcompleteness of the Picard functor $\spic_{\sx/S}$ from \cite[Proposition 19.2.4.7(1)]{lu-SAG}, $I(\sd)$ is an invertible $B$-module.  Therefore the associated ideal sheaf of $\sd$ is a line bundle over $\sx$.  
\end{proof}

\begin{lemma}
\label{lem:3b}
    The functor $\CDiv_{\se/R}$ is infinitesimally cohesive.  
\end{lemma}
\begin{proof}
    This follows from Lemma \ref{lem:def} and \cite[Proposition 16.3.2.1]{lu-SAG}.  
\end{proof}

\begin{lemma}
\label{lem:3c}
    The functor $\CDiv_{\se/R}$ is integrable.  
\end{lemma}
\begin{proof}
    Given a local Noetherian $\einf$-ring $R'$ which is complete with respect to its maximal ideal $\mathfrak{m}\subset\pi_0R'$, we need to prove that the inclusion functor $\Spf R'\hookrightarrow\Spec R'$ induces a homotopy equivalence 
    \[
        \Map_{\Fun(\CAlg^\cn,\,\cs)}(\Spec R',\CDiv_{\se/R})\to\Map_{\Fun(\CAlg^\cn,\,\cs)}(\Spf R',\CDiv_{\se/R}) 
    \]
    This can be deduced from the following result: Given a flat, proper, and separated spectral algebraic space $\sx$ locally almost of finite presentation over a connective local Noetherian $\einf$-ring $S$ which is complete with respect to its maximal ideal, we have an equivalence 
    \[
        \CDiv(\sx/S)\simeq\CDiv(\sx\times_{\Spet S}\Spf S) 
    \]
    Indeed, let $\Hilb(\sx/S)$ denote the full subcategory of $\SpDM_{/\sx}$ consisting of those $\sd\to\sx$, such that each $\sd\to\sx$ is a closed immersion and is flat, proper, and locally almost of finite presentation.  Then by the formal GAGA theorem \cite[Corollary 8.5.3.4]{lu-SAG} and the base-change properties of being flat, proper, and locally almost of finite presentation, we have $\Hilb(\sx/S)\simeq\Hilb(\sx\times_{\Spet S}\Spf S)$.  
    
    To prove the above equivalence for relative effective Cartier divisors, we need to further check that $\sd\to\sx$ associates a line bundle over $\sx$ if and only if $\sd\times_{\Spet S}\Spf S$ associates a line bundle over $\sx\times_{\Spet S}\Spf S$.  Note that the morphism $f\co\sx\times_{\Spet S}\Spf S\to\sx$ is flat by \cite[Corollary 7.3.6.9]{lu-SAG}, and so we have 
    \[
        \ci(\sd\times_{\Spet S}\Spf S)=\ci(f^*\sd)\simeq f^*\ci(\sd) 
    \]
    over the pullback square 
    \[
        \xymatrix{
		\sd\times_{\Spet S}\Spf S \ar[r] \ar[d] & \sd \ar[d] \\
	    \sx\times_{\Spet S}\Spf S \ar[r]^-f     & \sx
	}
    \]
    By \cite[proof of Proposition 19.2.4.7]{lu-SAG}, we have an equivalence 
    \[
        \QCoh(\sx/S)^{\aperf,\cn}\simeq\QCoh(\sx\times_{\Spet S}\Spf S)^{\aperf,\cn} 
    \]
    We need only restrict to the subcategories spanned by invertible objects via \cite[Proposition 2.9.4.2]{lu-SAG} to complete the proof.  
\end{proof}

\begin{lemma}
\label{lem:5}
    The functor $\CDiv_{\se/R}$ is locally almost of finite presentation over $\Spec R$.  
\end{lemma}
\begin{proof}
    By \cite[Definition 17.4.1.1(b)]{lu-SAG}, we need to prove that 
    \[
        \CDiv_{\se/R}\co\CAlg^\cn_R\to\cs,\quad R'\mapsto\CDiv(\se_{R'}/R') 
    \]
    commutes with filtered colimits when restricted to each $\tau_{\leq n}\CAlg^{\cn}_R$.  We note that $\CDiv(\se_{R'}/R')$ is a full subcategory of $\SpDM_{/(\se_{R'}\to\Spet R')}$ and first consider instead the functor 
    \[
        \Var^+\co\CAlg^\cn_R\to\widehat{\mathcal{C}\rm at}_\infty,\quad R'\mapsto\Var^+_{/(\se_{R'}\to\Spet R')} 
    \]
    where $\Var^+_{/(\se_{R'}\to\Spet R')}$ consists of diagrams 
    \[
        \xymatrix{ 
		\sd \ar[r] \ar[rd] & \se_{R'} \ar[d] \\
                               & \Spet R' 
	}
    \]
    such that $\sd\to\Spet R'$ is flat, proper, and locally almost of finite presentation.  Then by \cite[Proposition 19.4.2.1]{lu-SAG}, this functor commutes with filtered colimits when restricted to $\tau_{\leq n}\CAlg^{\cn}_R$.  It remains to verify that when $\{\sd_i\to\se_{i,R'}\}_{i\in I}$ are closed immersions and determine line bundles over $\{\se_{i,R'}\}$, $\varinjlim_{i\in I}\sd_i\to\varinjlim_{i\in I}\se_{i,R'}$ are closed immersions and determine line bundles over $\varinjlim_{i\in I}\se_{i,R'}$.  As we recalled earlier in this subsection, this follows from properties of closed immersions and the property of Picard functors that they are locally almost of finite presentation.  
\end{proof}

\begin{lemma}
\label{lem:4}
    The functor $\CDiv_{\se/R}$ admits a cotangent complex which is connective and almost perfect.  
\end{lemma}
\begin{proof}
    Let $S$ be a connective $R$-algebra, $\eta\in\CDiv_{\se/R}(S)$, and $M$ be a connective $S$-module.  We then have a pullback diagram 
    \[
        \xymatrix{
			F_\eta(M) \ar[r] \ar[d] & \CDiv_{\se/R}(S\oplus M) \ar[d] \\
			\{\eta\} \ar[r]         & \CDiv_{\se/R}(S) 
	}
    \]
    From this we obtain a functor 
    \[
        F_\eta\co\Mod_S\to\cs, \quad M\mapsto F_\eta(M) 
    \]

    We first need to prove that the above functor is corepresentable.  Here, $\eta$ is a morphism $\sd\to\se\times_RS$, and $\se\times_R(S\oplus M)$ is a square-zero extension of $\se\times_RS$.  Thus by the classification of first-order deformations \cite[Proposition 19.4.3.1]{lu-SAG}, the space of spectral algebraic spaces $\sd'$ which fit into the pullback diagram 
    \[
        \xymatrix{
		\sd \ar[r] \ar[d]^{\eta}          & \sd' \ar[d] \\
		\se\times_RS \ar[r] \ar[d]^{p}    & \se\times_R(S\oplus M) \ar[d] \\
		\Spet S \ar[r]                    & \Spet(S\oplus M) 
	}
    \]
    is equivalent to $\Map_{\QCoh(\sd)}\big(L_{\sd/(\se\times_RS)},\Sigma \eta^*(p^*M)\big)$.  Pushing forward along $p\circ\eta$, by \cite[Proposition 6.4.5.3]{lu-SAG}, we then have 
    \[
        \Map_{\QCoh(\sd)}\big(L_{\sd/(\se\times_RS)},\Sigma \eta^*(p^*M)\big)\simeq\Map_{\QCoh(\Spet S)}\big(\Sigma^{-1}p_+(\eta_+L_{\sd/(\se\times_RS)}),M\big) 
    \]
    By Lemma \ref{lem:def}, any such $\sd'\to\se\times_R(S\oplus M)$ is a closed immersion and determines a line bundle over $\se\times_R(S\oplus M)$.  Since the diagram 
    \[
        \xymatrix{
            \sd \ar[r] \ar[d] & \sd' \ar[d] \\
            \Spet S \ar[r]    & \Spet(S\oplus M) 
	}
    \]
    is a pullback square, $\sd'$ is a square-zero extension of $\sd$.  By \cite[Proposition 16.3.2.1]{lu-SAG}, $\sd'\to\Spet(S\oplus M)$ is flat, proper, and locally almost of finite presentation.  Combining these facts, we find that 
    \[
        F_\eta(M)=\Map_{\QCoh(\Spet S)}\big(\Sigma^{-1}p_+(\eta_+L_{\sd/(\se\times_RS)}),M\big) 
    \]
    Consequently, the functor $\CDiv_{\se/R}$ satisfies condition (a) from \cite[Example 17.2.4.4]{lu-SAG}.  Condition (b) therein follows from the compatibility of $(p\circ\eta)_+$, as a left adjoint of the functor $(p\circ\eta)^*$, with base change (cf.~\cite[Construction 6.4.5.1 and Proposition 6.4.5.3]{lu-SAG}).  Therefore the functor $\CDiv_{\se/R}$ admits a cotangent complex $L_{\CDiv_{\se/R}}$ satisfying 
    \[
        \eta^*L_{\CDiv_{\se/R}}=\Sigma^{-1}p_+(\eta_+L_{\sd/(\se\times_RS)}) 
    \]
    Since the quasi-coherent sheaf $L_{\sd/(\se\times_RS)}$ is connective and almost perfect \cite[Proposition 17.1.5.1(3)]{lu-SAG}, the $S$-module $\Sigma^{-1}p_+(\eta_+L_{\sd/(\se\times_RS)})$ is $(-1)$-connective.  

    Next, we show that $L_{\CDiv_{\se/R}}$ is almost perfect.  This follows from \cite[17.4.2.2]{lu-SAG} and Lemma \ref{lem:5}.  
    
    Finally, we show that it is connective.  As above, let $S$ be a connective $R$-algebra and $\eta\in\CDiv_{\se/R}(S)$.  We need to prove that $M_\eta\coloneqq\eta^*L_{\CDiv_{\se/R}}\in\Mod_S$ is connective.  We already knew that $M_\eta$ is $(-1)$-connective and almost perfect.  In particular, the homotopy group $\pi_{-1}M_\eta$	is a finitely generated $\pi_0S$-module.  To prove that it in fact vanishes, by Nakayama's lemma, we note that this is equivalent to proving that 
    \[
        \pi_{-1}(\kappa\otimes_{\pi_0S}M_\eta)\simeq\Tor^{\pi_0S}_0(\kappa,\pi_{-1}M_\eta) 
    \]
    equals 0 for every residue field $\kappa$ of $\pi_0S$.  Thus we may replace $S$ by $\kappa$ and assume $\kappa$ is an algebraically closed field.  

    Let $A=\kappa[\epsilon]/(\epsilon^2)$.  Unwinding the definitions, we find that the dual space $\Hom_\kappa(\pi_{-1}M_\eta,\kappa)$ can be identified with the set of automorphisms of the base change $\eta_A$ such that they restrict to be the identity of $\eta$.  It remains to prove that this set is trivial.  This boils down to the following assertion in classical algebraic geometry: 
    \begin{quote}
        Let $X/\kappa$ be a scheme, $L$ be a line bundle over $X$, and assume $L_A$ is also a line bundle over $X_A$.  If $f$ is an automorphism of $L_A$ such that $f|L$ is the identity on $L$, then $f$ is the identity.  
    \end{quote}
    This can be proved, mutatis mutandis, as in the last part of \cite[proof of Proposition 2.2.6]{lu-EC1}.  
\end{proof}

\section{Level structures for spectral abelian varieties}
\label{sec:level}

For spectral Deligne--Mumford stacks, Theorem \ref{thm:cdiv} gives the relative representability (with respect to a fixed $\se/R$) of relative effective Cartier divisors (over $\Spet R$).  Their analogues in classical algebraic geometry are crucial to Drinfeld's approach to arithmetic moduli of elliptic curves with level structure over $\bz$, as developed in \cite{katz1985arithmetic}, which applies nicely at primes dividing the level.  In this section, we define level structures on spectral abelian varieties and related objects from effective Cartier divisors.  The applications we aim at are of a similar nature to those considered by the earlier authors, i.e., incorporating ramification or regardless of failure of \'etaleness, which we will discuss in the next two sections.  

\subsection{Level structures on elliptic curves}
\label{subsec:elevel}

Let $C$ be a one-dimensional smooth commutative group scheme over a base scheme $S$, and $A$ be an abstract finite abelian group.  Recall from \cite[1.5.1]{katz1985arithmetic} that a homomorphism 
\[
    \phi\co A\to C(S) 
\]
of abstract groups is said to be an {\em $A$-structure on $C/S$}, if the effective Cartier divisor $\sum_{a\in A}[\phi(a)]$ is a subgroup scheme of $C/S$.  

The following result gives the relative representability of moduli problems of level structures.  

\begin{proposition}[{\cite[Proposition 1.6.2]{katz1985arithmetic}}]
    Let $C$ be a one-dimensional smooth commutative group scheme over $S$.  Then the functor 
    \[
        \Level^A_{C/S}\co\Sch_S\to\Set,\quad T\mapsto\text{the set of level-$A$ structures on $C_T/T$} 
    \]
    is represented by a closed subscheme of $\Hom_{{\rm Grp}/S}(A,C)$.  
\end{proposition}

\begin{definition}
\label{def:elevel}
    Let $R$ be a connective $\einf$-ring and $\se/R$ be a spectral elliptic curve.  A {\em (derived) level-$A$ structure on $\se$} is a pair $(\sd,\phi)$, where $\sd\to\se$ is a relative effective Cartier divisor, and $\phi\co A\to \se^\heartsuit(\pi_0R)$ is an $A$-structure on $\se^\heartsuit/\pi_0R$ as above, such that the underlying morphism $\sd^\heartsuit\to\se^\heartsuit$, necessarily a closed immersion, equals the inclusion of the associated relative effective Cartier divisor $\sum_{a\in A}[\phi(a)]$ into $\se^\heartsuit$.  We denote by $\Level(A,\se/R)$ the $\infty$-category of level-$A$ structures on $\se/R$, whose objects can be viewed as relative effective Cartier divisors satisfying an extra property.  
\end{definition}

Given a spectral elliptic curve $\se/R$, the $\infty$-category $\Level(A,\se/R)$ is an $\infty$-groupoid, since it is a full subcategory of $\CDiv(\se/R)$, which is an $\infty$-groupoid (see Remark \ref{rmk:Kan}).  

We note that derived level structures are stable under base change, as follows.  

\begin{lemma}
    Let $\se/R$ be a spectral elliptic curve, $\ss=\Spet R$, and $(\sd,\phi)$ be a level structure.  Suppose that $\st\to\ss$ is a morphism of nonconnective spectral Deligne--Mumford stacks.  Then the induced pair $(\sd_\st,\phi_\st)$ is a level structure on $\se_\st/\st$.  
\end{lemma}
\begin{proof}
    The induced closed immersion $\sd_\st\to\se_\st$ is a relative effective Cartier divisor by Lemma \ref{lem:bc}.  It remains to check that $\phi_\st\co A\to\se^\heartsuit(\st^\heartsuit)=\se_\st^\heartsuit(\st^\heartsuit)$ is a classical level structure, so that $\sd_\st^\heartsuit$ is the associated classical relative effective Cartier divisor.  This follows from the base-change property of classical level structures observed in \cite[Section 1.5.1]{katz1985arithmetic}.  
\end{proof}

We next recall a result on when a divisor becomes a (finite flat) subgroup.  

\begin{proposition}[{\cite[Corollary 1.3.7]{katz1985arithmetic}}]
    Given a smooth curve $C/S$ which is a group scheme over a scheme $S$ along with a relative effective Cartier divisor $D$ of $C$, there exists a closed subscheme $Z$ of $S$ with the property that, for any $T\to S$, $D_T$ is a subgroup of $C_T$ if and only if $T\to S$ factors through $Z$.  
\end{proposition}

Here we have an analogous incidence object for the relation ``$\sd^\heartsuit$ is a subgroup.''  

\begin{lemma}
\label{lem:incidence}
    Let $\se/\Spet R$ be a spectral elliptic curve and $\sd\to\se$ be a relative effective Cartier divisor.  Then there exists a closed spectral Deligne--Mumford substack $\Spet B\subset\Spet R$ satisfying the following universal property: 
    \begin{quote}
        Given any $R'\in\CAlg^\cn_R$, $\sd_{R'}^\heartsuit$ is a subgroup of $\se_{R'}^\heartsuit$ if and only if $R\to R'$ factors through $B$.
    \end{quote}
\end{lemma}
\begin{proof}
    By the preceding proposition, if $\sd_{R'}^\heartsuit/\pi_0R'$ is a subgroup of $\se_{R'}^\heartsuit/\pi_0R'$, the morphism $\Spec\pi_0R'\to\Spec\pi_0R$ must factor through a closed subscheme $Z=\Spec B_0$ of $\Spec\pi_0R$. This corresponds to a closed spectral subscheme $\Spet B$ of $\Spet R$.  In fact, since the map $R\to R'$ satisfies that $\pi_0R\to\pi_0R'$ factors through $\pi_0R/I$ for some ideal $I$ of $\pi_0R$, we obtain a factorization of $R\to R'$ through $\Gamma_IR$ (see \cite[Chapter 7, esp.~Definition 7.1.2.1]{lu-SAG} for details about $I$-nilpotent $R$-modules).  Conversely, suppose that $R\to R'$ factors through $B$.  Then $\so_{\Spet R'}$ vanishes on some $I\subset\pi_0R$.  In other words, we have that $\pi_0R\to\pi_0R'$ factors through $\pi_0R/\sqrt{I}$.  This is equivalent to $\Spec\pi_0R'\to\Spec\pi_0R$ factoring through $\Spec\pi_0R/I=\Spec B_0=Z$, and so $\sd_{R'}^\heartsuit$ is a subgroup of $\se_{R'}^\heartsuit$.  
\end{proof}

The following is our main result in this subsection on relative representability of level structures over the spectral moduli stack of spectral elliptic curves.  

\begin{theorem}
\label{thm:elevel}
    Let $\se/R$ be a spectral elliptic curve and $A$ be an abstract finite abelian group.  Then the functor 
    \[
        \Level^A_{\se/R}\co\CAlg_R^\cn\to\cs,\quad R'\mapsto\Level(A,\se_{R'}/R') 
    \]
    is represented by a closed substack $\ss(A)$ of $\CDiv_{\se/R}$.  Moreover, $\ss(A)=\Spet\cp_{\se/R}$ for some $\einf$-ring $\cp_{\se/R}$, which is locally almost of finite presentation over $R$.  
\end{theorem}
\begin{proof}
    By definition, the functor $\Level^A_{\se/R}$ is a subfunctor of the representable functor $\CDiv_{\se/R}$ from Theorem \ref{thm:cdiv}.  In view of Lemma \ref{lem:incidence}, we consider a spectral Deligne--Mumford stack $\GrpCDiv$ defined by the pullback diagram of spectral Deligne--Mumford stacks 
    \[
        \xymatrix{
	   \GrpCDiv_{\se/R} \ar[d] \ar[r] & \CDiv_{\se/R} \ar[d] \\
          \Spet B \ar[r]                 & \Spet R 
	}
    \]
    where $B$ is associated to the universal object $\sd_\univ\to\se\times_R\CDiv_{\se/R}$.  We verify that $\GrpCDiv_{\se/R}$ valued on an $R$-algebra $R'$ is the space of relative effective Cartier divisors $\sd$ of $\se_{R'}$ such that $\sd^\heartsuit$ is a finite flat subgroup of $\se_{R'}^\heartsuit$.  
    
    Moreover, there is a clopen substack $\CDiv^A_{\se/R}$ of $\GrpCDiv_{\se/R}$ whose value on an $R$-algebra $R'$ is the space of relative effective Cartier divisors $\sd$ of $\se_{R'}$ such that $\sd^\heartsuit$ is a subgroup of $\se_{R'}^\heartsuit$ finite locally free over $\pi_0R'$ of rank equal to $\#A$.  We then evoke \cite[Proposition 1.6.5]{katz1985arithmetic} and obtain $\ss(A)$ representing $\Level^A_{\se/R}$ as a closed substack of $\CDiv^A_{\se/R}$ similarly as in Lemma \ref{lem:incidence}.  

    To prove the remaining statement, we consider the morphism $\ss(A)\to\Spet R$, both being spectral algebraic spaces.  By \cite[Remark 5.2.0.2]{lu-SAG}, a morphism between spectral algebraic spaces is finite if and only if its underlying morphism between ordinary algebraic spaces is finite in the sense of classical algebraic geometry.  Thus we need only prove that $\ss(A)^\heartsuit$ is finite over $\Spec\pi_0R$.  This is precisely the classical case: $\ss(A)^\heartsuit$ is the representing object of the classical level-$A$ structures, which is a finite $\pi_0R$-scheme of finite presentation by \cite[Corollary 1.6.3]{katz1985arithmetic}.  
\end{proof}

\subsection{Level structures on \textit{p}-divisible groups}
\label{subsec:plevel}

Before we move on and introduce derived level structures for spectral $p$-divisible groups, let us first recall some classical facts needed about level structures of commutative finite flat group schemes.  

\subsubsection{Classical finite flat group schemes}

Let $S$ be a scheme and $X/S$ be a finite flat $S$-scheme of finite presentation and rank $N$.  It can be proved that $X/S$ is finite locally free of rank $N$.  This means that for every affine scheme $\Spec R\to S$, the pullback scheme $X\times_S\Spec R$ over $\Spec R$ has the form $\Spec R'$, where $R'$ is an $R$-algebra which is locally free of rank $N$.  For an element $f\in R'$ acting on $R'$ by multiplication, define an $R$-linear endomorphism of $R'$.  Because $R'$ is locally free of rank $N$, multiplication by $f$ has a characteristic polynomial 
\[
    \det(T-f)=T^N-{\rm trace}(f)\,T^{N-1}+\cdots+(-1)^N{\rm norm}(f) 
\]
Recall the following definition from \cite[1.8.2]{katz1985arithmetic}.  Let $\{P_1, \ldots, P_N\}$ be a set of $N$ points not necessarily distinct in $X(S)$.  We call it a {\em full set of sections of $X/S$} if one of the following two equivalent conditions is satisfied: 
\begin{itemize}
    \item For any $\Spec R\to S$ and $f\in R'=H^0(X_R,\so)$, we have 
    \[
        \det(T-f)=\prod_{i=1}^N\big(T-f(P_i)\big) 
    \]
    
    \item For any $\Spec R\to S$ and $f\in R'=H^0(X_R,\so)$, we have 
    \[
        {\rm norm}(f)=\prod_{i=1}^Nf(P_i) 
    \]
\end{itemize}

Given $N$ not necessarily distinct points $P_1, \ldots, P_N$ in $X(S)$, we have a morphism $\so_X\to\bigotimes_{i}(P_i)_*(\so_S)$ of sheaves over $X$.  It is not hard to see that this morphism is surjective and defines a closed subscheme $D$ of $X$ which is flat and proper over $S$.  Thus, given an abstract finite abelian group $A$ and a map $\phi\co A\to X(S)$ of sets, we can define a closed subscheme $D$ of $X$ by the sheaf $\bigotimes_{a\in A}\phi(a)_*\so_S$.  

\begin{lemma}
\label{lem:open}
    Given a finite flat $S$-scheme $Z$ of finite presentation, $\Hom(A,Z)$ is an open subscheme of $\Hilb_{Z/S}$.  
\end{lemma}
\begin{proof}
    Let $T\to S$ be an $S$-scheme.  For any $D\to Y\coloneqq T\times_SZ$ in $\Hilb(Y)$, we need to prove that the set of points $t\in T $ over which $D_t\to Y_t$ comes from the closed subscheme associated to $\phi\co A\to Z(T)=Y(T)$ is open in $T$.  Since $D$ is the closed subscheme defined by $\so_Y\to\so_D$, if $D_t$ comes from $\so_{Y,t}\to\bigotimes_{a\in A}\phi(a)_*\so_{T,t}$, then by the definition of a stalk, there exists an open subset $U$ of $T$ such that $t\in U$ and $D_U$ is defined by $\so_Y|_U\to\bigotimes_{a\in A}\phi(a)_*\so_T|_U$.  
\end{proof}

Suppose that $G/S$ is a finite flat commutative group scheme of finite presentation and $A$ is an abstract finite abelian group of order $N$.  Let $K$ be a finite flat $S$-subgroup-scheme of $G$ locally free of rank $N$, and $\phi\co A\to G(S)$ be a homomorphism landing in $K(S)$.  Recall from \cite[Remark 1.10.10]{katz1985arithmetic} that the pair $(K,\phi)$ is called an {\em $A$-structure on $G/S$}, if the $N$ points $\phi(a), a\in A$ form a full set of sections of $K$.  

\begin{lemma}
    Suppose that $G/S$ is a finite flat commutative group scheme of finite presentation and $K\subset G$ is a closed subscheme which is finite flat and finite presentation.  Then there exists a closed subscheme $Z\subset S$ such that given any morphism of schemes $T\to S$, $K_T$ is a subgroup scheme of $G_T$ if and only if the morphism $T\to S$ factors through $Z$.  
\end{lemma}
\begin{proof}
    This is an analogue of \cite[Corollary 1.3.7]{katz1985arithmetic} for finite flat group schemes.  Following the proof strategy there, we need only prove: Given finite flat closed subschemes $K_1,K_2$ of $G$, there exists a closed subscheme $Z\subset S$ such that given any morphism of schemes $T\to S$, $(K_1)_T$ is a closed subscheme of $(K_2)_T$ if and only if the morphism $T\to S$ factors through $Z$ (cf.~\cite[Lemma 1.3.4\,(1)]{katz1985arithmetic}).  

    Since we consider the case of finite flat group schemes of finite presentation and the question is local on $S$, we are reduced to proving: 
    \begin{quote}
        Let $B$ be a finite free $A$-algebra, and $\Spec B/I_1,\Spec B/I_2$ be two closed subschemes of $\Spec B$ such that $B/I_1$ and $B/I_2$ are also free.  Then there exits a closed subscheme $\Spec W$ of $\Spec A$ such that given any $A\to A'$, $\Spec(B/I_1\otimes_AA')$ is a closed subscheme of $\Spec(B/I_2\otimes_AA')$ if and only if $A\to A'$ factors through $W$.  
    \end{quote}
    Let $\bar{1}\in B/I_2$ be the identity.  Since $B/I_1$ is a free $A$-module of finite rank, the image of $\bar{1}$ under the map $B/I_2\to B/I_2\otimes_AB/I_1$ can be written as $\sum^{d}_{i=1}r_ie_i$, where $d$ is the rank and $\{e_i\}_{1\leq i\leq d}$ is an $A$-basis of $B/I_2\otimes_AB/I_1$.  It is not hard to see that $V(\{r_1,\cdots,r_d\})$ is the desired closed subscheme of $\Spec A$.  
\end{proof}

\begin{proposition}
\label{prop:astr}
    Hypotheses and notations as above, the functor $A\text{-}{\rm Str}(G/S)$ on $S$-schemes defined by 
    \[
        T\mapsto\big\{\big(K\subset G_T,\,\phi\co A\to G(T)\big)\,\big|\,\text{$(K,\phi)$ is an $A$-structure on $G_T$}\big\} 
    \]
    is representable by a finite $S$-scheme of finite presentation.  
\end{proposition}
\begin{proof}
    This is a variant of \cite[Lemma 1.10.11 and Proposition 1.10.13\,(1)]{katz1985arithmetic}.  Let us proceed in 3 steps.  
    
    First, the functor 
    \[
        T\mapsto\{D_T\subset G_T\,|\,\text{$D_T$ is a closed subscheme finite flat over $T$ of rank $N$}\} 
    \]
    is representable by a finite $S$-scheme $\Hilb^N_{G/S}$ (a Grassmannian).  
    
    Second, applying the preceding lemma to the universal example over $\Hilb^N_{G/S}$, we obtain a finite $S$-scheme $Z$ classifying finite flat subgroup schemes of $G$ locally free of rank $N$.  

    Third, given such a subgroup scheme $K\subset G$, observe that the functor 
    \[
        T\mapsto\{\phi\co A\to G(T)\,|\,\text{$(K_T,\phi)$ is an $A$-level structure on $G_T$}\} 
    \]
    is equivalent to the functor 
    \[
        T\mapsto\{\phi\co A\to K(T)\,|\,\text{$\phi$ is an $A$-generator of $K_T$}\}
    \]
    (cf.~\cite[Remark 1.10.10 and 1.10.5]{katz1985arithmetic}).  Since the latter is representable by a finite $S$-scheme of finite presentation by \cite[Proposition 1.10.13\,(1)]{katz1985arithmetic}, we further apply this representability to the universal example $K_\univ\subset G_Z$ to complete the proof.    
\end{proof}

\subsubsection{Spectral finite flat group schemes}

Let $R$ be a connective $\einf$-ring and $G$ be a commutative finite flat group scheme over $R$.  By the definition of finite-flatness, we have $G=\Spet B$ for a finite flat $R$-algebra $B$ \cite[Definition 6.1.2]{lu-EC1}.  We let $\Hilb(G/R)$ denote the full subcategory of $\SpDM_{/G}$ spanned by those $\sd\to G$ such that $\sd\to G$ is a closed immersion of spectral Deligne--Mumford stacks and that the composite $\sd\to G\to\Spet R$ is flat, proper, and locally almost of finite presentation.  Then $\Hilb(G/R)$ is equivalent to the $\infty$-category of diagrams of $\einf$-rings 
\[
    \xymatrix{
        R \ar[rr] \ar[rd] &   & B \ar[ld] \\
	                   & S & \\
    }
\]
such that $S$ is flat, proper, and locally almost of finite presentation over $R$ subject to certain additional conditions.  It is not hard to see that $\Hilb(G/R)$ is a Kan complex (cf.~Remark \ref{rmk:Kan}), so that we can define a functor 
\[
    \Hilb_{G/R}\co\CAlg^\cn_R\to\cs,\quad R'\mapsto\Hilb(G_{R'}/R') 
\]
The representability of this functor is a special case of \cite[Theorem 8.3.3]{lurie2004derived}, which we record below.  Like that theorem and Theorem \ref{thm:cdiv}, it can be deduced from the spectral Artin representability theorem \ref{thm:Lurie}.  

\begin{theorem}[Lurie]
    Suppose that $G$ is a commutative finite flat group scheme over a connective $\einf$-ring $R$.  Then $\Hilb_{G/R}$ is representable by a spectral Deligne--Mumford stack which is locally almost of finite presentation over $R$.  
\end{theorem}

\begin{corollary}
    Hypotheses and notations as above, for each positive integer $N$, there exists a substack  $\Hilb^N_{G/R}$ of $\Hilb_{G/R}$ such that given any $R'$ in $\CAlg^\cn_R$, the space $\Hilb^N_{G/R}(R')\eqqcolon\Hilb^N(G_{R'}/R')\subset\Hilb(G_{R'}/R')$ consists of those $\sd\to G_{R'}$ locally free of rank $N$ over $R'$.  
\end{corollary}

\begin{definition}
\label{def:glevel}
    Let $R$ be a connective $\einf$-ring and $G$ be a spectral commutative finite flat group scheme over $R$.  Given an abstract finite abelian group $A$ of order $N$, a {\em level-$A$ structure on $G$} is a pair $(\sd,\phi)$, where $i\co\sd\to G$ is an object in $\Hilb^N(G/R)$ and $\phi\co A\to G^\heartsuit(\pi_0R)$ is a homomorphism, such that $(\sd^\heartsuit,\phi)$ is an $A$-structure in the sense of \cite[Remark 1.10.10]{katz1985arithmetic}, i.e., $\pi_0i_*\so_\sd=\bigotimes_{a\in A}\phi(a)_*\so_{\Spec\pi_0R}$.  We denote by $\Level(A,G/R)$ the $\infty$-category of level-$A$ structures on $G/R$.  
\end{definition}

\begin{remark}
    Given a level-$A$ structure $(\sd,\phi)$ on $G$, $\sd$ is locally free of rank $N$ over $R$, since $\sd\to G$ is a closed immersion, $\sd\to\Spet R$ is flat, and $\pi_0i_*\so_\sd=\bigotimes_{a\in A}\phi(a)_*\so_{\Spec\pi_0R}$.  The last identity also ensures the group structure on $\sd^\heartsuit$.  
\end{remark}

\begin{remark}
    Comparing Definition \ref{def:glevel} with Definition \ref{def:elevel}, we see that \cite[Proposition 1.10.6]{katz1985arithmetic} establishes an equivalence between the two definitions in the classical case when $G^\heartsuit/\pi_0R$ is embeddable as a closed subscheme of an elliptic curve $E^\heartsuit/\pi_0R$.  Thus these two definitions are compatible if the spectral group scheme $G/R$ is embeddable as a closed substack of a spectral elliptic curve $E/R$.  
\end{remark}

To prove representability for the functor of level-$A$ structures, we present a second result concerning the existence of incidence spectral Deligne--Mumford stacks (cf.~Lemma \ref{lem:incidence}).  
\begin{lemma}
    Let $G/R$ be a spectral commutative finite flat group scheme over a connective $\einf$-ring $R$.  Let $A$ be an abstract finite abelian group of order $N$.  Given an object $\sd\to G$ in $\Hilb^N(G/R)$, there exists an $\einf$-ring $W$ satisfying the following universal property: 
    \begin{quote}
        For any $R\to R'$ in $\CAlg_R^\cn$, $\sd_{R'}$ supports a level-$A$ structure on $G_{R'}$ if and only if $R\to R'$ factors through $W$.  
    \end{quote}
\end{lemma}
\begin{proof}
    Given $R'$ in $\CAlg^\cn_R$, it is clear that $\sd_{R'}$ is in $\Hilb(G_{R'}/R')$.  For $\sd_{R'}$ to support a level-$A$ structure as in Definition \ref{def:glevel}, $\Spec\pi_0R'\to\Spec\pi_0R$ must factor through $\Hom(A,G^\heartsuit)$, which is open in $\Hilb_{G^\heartsuit/\pi_0R}$ by Lemma \ref{lem:open}.  Thus $\pi_0R\to\pi_0R'$ factors through $B_0$ for some localization $B_0$ of $\pi_0R$.  This lifts to a factorization of $R \to R'$ through an $\einf$-ring $B$, which is a localization of $R$ with $\pi_0B\simeq B_0$ (see \cite[Remark 1.1.4.2]{lu-SAG}).  
    
    By now, along the map $\Spet R'\to\Spet B$, we already have 
    $i\co\sd_{R'}\to G_{R'}$ in $\Hilb^N(G_{R'}/R')$ and a map $\phi\co A\to G^\heartsuit(\pi_0R')$ associated with $\pi_0i_*\so_{\sd_{R'}}$.  For $(\sd_{R'},\phi)$ to be a level-$A$ structure, 
    \[
        \bigotimes_{a\in A}\phi(a)_*\so_{\Spec\pi_0R'}\to\pi_0i_*\so_{\sd_{R'}} 
    \]
    needs to be an isomorphism, i.e., the $N$ points $\phi(a), a\in A$ must form a full set of sections of $\sd_{R'}^\heartsuit$.  By \cite[Proposition 1.9.1]{katz1985arithmetic}, $\Spec\pi_0R'\to\Spec\pi_0B$ must then factor through a closed subscheme of $\Spec\pi_0B$.  Thus $\pi_0B\to\pi_0R'$ factors through $W_0=B_0/I$ for some ideal $I$.  This lifts to a factorization of $B\to R'$ through the $\einf$-ring $W=\Gamma_I(B)$, as desired.  
    
    We can show the converse by a similar argument to the one in the proof of Lemma \ref{lem:incidence}.  
\end{proof}	

\begin{proposition}
\label{prop:glevel}
    Suppose that $G$ is a spectral commutative finite flat group scheme over a connective $\einf$-ring $R$ and $A$ is an abstract finite abelian group.  Then the functor 
    \[
        \Level^A_{G/R}\co\CAlg_R^\cn\to\cs,\quad R'\mapsto\Level(A,G_{R'}/R') 
    \]
    is representable by an affine spectral Deligne--Mumford stack $\ss(A)=\Spet\cp_{G/R}$.  
\end{proposition}
\begin{proof}
    We first prove the representability.  By definition, the functor $\Level^A_{G/R}$ is a subfunctor of the representable functor $\Hilb^N_{G/R}$, where $N=\#A$.  In view of the previous lemma and its proof, let us consider from right to left the consecutive pullbacks of universal objects 
    \[
	\xymatrix{
            (G,\text{univ.~level-$A$ str.~on $G$}) \ar[d] \ar[r] & (G,\sd_\univ,\phi_\univ) \ar[d] \ar[r] & (G,\sd_\univ) \ar[d] \ar[r] & G \ar[d] \\
            W \ar[r] & B \ar[r] & \Hilb^N_{G/R} \ar[r] & \Spet R 
	}
    \]
    It is straightforward to verify that $\ss(A)\coloneqq W $ valued on an $R$-algebra $R'$ is precisely the space of level-$A$ structures on $G_{R'}$.  

    For the affineness property, we need to prove that $\ss(A)$ is finite over $R$ in the sense of spectral algebraic geometry.  By \cite[Remark 5.2.0.2]{lu-SAG}, a morphism between spectral algebraic spaces is finite if and only if its underlying morphism between ordinary algebraic spaces is finite in the sense of classical algebraic geometry.  Thus we need only prove that $\ss(A)^\heartsuit$ is finite over $\pi_0R$.  This follows from Proposition \ref{prop:astr}.  
\end{proof}

\subsubsection{Spectral $p$-divisible groups}
\label{subsubsec:pdiv}

Given an $\einf$-ring $R$, let $\FFG(R)$ denote the $\infty$-category of spectral commutative finite flat group schemes over $R$.  Let $X\co(\Ab_\fin^p)^\op\to\FFG(R)$ be a spectral $p$-divisible group of height $h$ over an $\einf$-ring $R$ (see \cite[Definition 6.5.1]{lu-EC1} and cf.~\cite[Definition 2.0.2]{lu-EC2}).  For each nonnegative integer $r$, we write $X[p^r]$ for the image of $\bz/p^r\bz$ under $X$, which is a degree-$(p^r)^h$ spectral commutative finite flat group scheme over $R$.  


\begin{definition}
\label{def:plevel}
    Let $\bG$ be a spectral $p$-divisible group of height $h$ over a connective $\einf$-ring $R$.  A {\em level-$(\bz/p^r\bz)^h$ structure on $\bG$} is a level-$(\bz/p^r\bz)^h$ structure on $\bG[p^r]$ as in Definition \ref{def:glevel}.  We let $\Level(r,\bG/R)$ denote the $\infty$-groupoid of level-$(\bz/p^r\bz)^h$ structures on $\bG/R$.  
\end{definition}

\begin{remark}
\label{rmk:derived}
    Recall that a level-$(\bz/p^r\bz)^h$ structure on $\bG[p^r]$ is a pair $(\sd,\phi)$, where $\sd\subset\bG[p^r]$ is a finite flat closed substack of rank $(\bz/p^r\bz)^h$ over $R$, and $\phi\co(\bz/p^r\bz)^h\to\bG[p^r]^\heartsuit(\pi_0R)$ is a homomorphism such that $(\sd^\heartsuit,\phi)$ is a level-$(\bz/p^r\bz)^h$ structure on $\bG[p^r]^\heartsuit/\pi_0R$.  Given such a structure $(\sd,\phi)$, since $\bG[p^r]$ is locally free of rank $(p^r)^h$ over $R$, the rank of $\sd^\heartsuit$ over $\pi_0R$ equals that of $\bG[p^r]^\heartsuit$.  Since $\sd^\heartsuit$ is a closed subscheme of $\bG[p^r]^\heartsuit$, they must then equal.  Thus $\phi$ is a {\em $(\bz/p^r\bz)^h$-generator of $\bG[p^r]^\heartsuit(\pi_0R)$} in the sense of \cite[1.10.5]{katz1985arithmetic}.  Note that as spectral Deligne--Mumford stacks, even though $\sd$ has the same rank as $\bG[p^r]$, they are not equivalent, since closed immersions in spectral algebraic geometry are not categorical monomorphisms (see \cite[Warning 6.2.3]{lu-EC1}).  For this reason, we do not introduce the concept of $A$-generators when discussing derived level structures, and this is where the higher homotopical information of {\em derived} level structures resides.  
\end{remark}

\begin{theorem}
\label{thm:plevel}
    Let $\bG$ be a spectral $p$-divisible group of height $h$ over a connective $\einf$-ring $R$.  Then the functor
    \[
        \Level^r_{\bG/R}\co\CAlg_R^\cn\to\cs,\quad R'\mapsto\Level(r,\bG_{R'}/R') 
    \]
    is representable by an affine spectral Deligne--Mumford stack $\ss(r) =\Spet\cp^r_{\bG/R}$.  
\end{theorem}
\begin{proof}
    We just notice that by the definition of a spectral $p$-divisible group, $\bG[p^r]$ is a spectral commutative finite flat group scheme.  Thus the theorem follows from Proposition \ref{prop:glevel} above about general spectral commutative finite flat group schemes.  
\end{proof}

\begin{remark}
\label{rmk:nonconnective}
    Our derived level structure functor is defined over $\CAlg^\cn$.  More generally, in view of \cite[Remark 6.1.3]{lu-EC1}, we can define such structures on $\bG/R$ where $R$ is not necessarily connective.  Let 
    \[
        \Level(r,\bG/R)\coloneqq\Level(r,\tau_{\geq0}\bG/\tau_{\geq0}R) 
    \]
    The corresponding functor $\Level^r_{\bG/R}\co\CAlg_R\to\cs$ is also representable.  This will be useful in Section \ref{subsec:LT} when we consider {\em oriented} spectral $p$-divisible groups.  The authors thank Yuchen Wu for bringing this point to their attention.  
\end{remark}

\subsubsection{Non-full level structures}
\label{subsubsec:nonfull}

So far we have treated only full level structures on commutative finite flat group schemes.  Here let us consider more general level structures, such as those relevant for power operations in Morava E-theories (see Section \ref{subsec:po}).  

\begin{definition}
\label{def:nonfull}
    Suppose that $G$ is a spectral commutative finite flat group scheme over a connective $\einf$-ring $R$.  We let $\Level_1(r,G/R)$ denote the $\infty$-groupoid of derived level-$(\bz/p^r\bz)$ structures on $G/R$.  We let $\Level_0(r,G/R)$ denote the $\infty$-groupoid of equivalence classes $(\sd,\phi)$ in $\Level_1(r,G/R)$ where two objects $(\sd,\phi)$ and $(\sd',\phi')$ are equivalent if the scheme-theoretic image of $\sd^\heartsuit$ under $\phi$ and that of $(\sd')^\heartsuit$ under $\phi'$ equal in $G^\heartsuit/\pi_0R$.  
\end{definition}

\begin{remark}
    Our notations above are intended to be consistent with the standard ones $\Gamma_1(p^r)$, $\Gamma_0(p^r)$, etc.~for the classical moduli problems.  We drop the prime $p$ altogether for readability when the level appears in a superscript, as the results here apply to all primes.  
\end{remark}

\begin{proposition}
    Hypotheses and notations as above, for each $i\in\{0,1\}$, the functor 
    \[
        \Level^{i,r}_{G/R}\co\CAlg^\cn_R\to\cs,\quad R'\to\Level_i(r,G_{R'}/R') 
    \]
    is representable by an affine spectral Deligne--Mumford stack $\Spet\cp^{i,r}_{G/R}$.  
\end{proposition}
\begin{proof}
    The statement for $i=1$ is a direct consequence of the more general Proposition \ref{prop:glevel}.  For $i=0$, we just notice that the classical level structure functor $\Level^{0,r}_{G^\heartsuit/\pi_0R}$ is representable by a closed subscheme of the Grassmannian of all rank-$p^r$ quotients of $G^\heartsuit[p^r]$ (cf.~\cite[Theorem 6.6.1 and proof of Proposition 6.5.1]{katz1985arithmetic}).  By an argument analogous to that for the case of full level structures, we obtain the desired result.  
\end{proof}

\begin{remark}
\label{rmk:pdiv}
    From the above proposition, we obtain analogous representability results for spectral $p$-divisible groups as in Section \ref{subsubsec:pdiv}.  
\end{remark}

\section{Moduli problems of derived level structures}
\label{sec:moduli}

In this section, we apply the derived level structures and their representability results from Section \ref{sec:level} and discuss several related spectral moduli problems.  

\subsection{Spectral elliptic curves with level structure}
\label{subsec:mella}

In Section \ref{subsec:elevel}, given an abstract finite abelian group $A$, we defined level-$A$ structures for spectral elliptic curves (Definition \ref{def:elevel}) and showed their representability relative to an object $\se/R$ (Theorem \ref{thm:elevel}).  Here, we consider their absolute representability (cf.~\cite[Sections 4.2--4.3]{katz1985arithmetic}).  

There exists a spectral Deligne--Mumford stack $\cm_\El$ whose functor of points is 
\[
    \cm_\El\co\CAlg^\cn\to\cs,\quad R\mapsto\cm_\El(R) 
\]
where $\cm_\El(R)=\El(R)^\simeq$ is the underlying $\infty$-groupoid of the $\infty$-category of spectral elliptic curves over $R$ \cite[Theorem 2.4.1]{lu-EC1}.  

In classical algebraic geometry, we have the Deligne--Mumford stack of (ordinary) elliptic curves, which can be viewed as a spectral Deligne--Mumford stack 
\[
    \cm^\cl_\El\co\CAlg^\cn\to\cs,\quad R\mapsto\cm^\cl_\El(\pi_0R) 
\]
where $\cm^\cl_\El(\pi_0R)$ is the groupoid of elliptic curves over the commutative ring $\pi_0R$.  

Moreover, if $A$ equals $\bz/N\bz$ or $(\bz/N\bz)^2$ with $N\geq1$ an integer, we have the Deligne--Mumford stack of elliptic curves with level-$A$ structures, which can also be viewed as a spectral Deligne--Mumford stack 
\[
    \cm^{\cl,A}_\El\co\CAlg^\cn\to\cs,\quad R\mapsto\cm^{\cl,A}_\El(\pi_0R) 
\]
where $\cm^{\cl,A}_\El(\pi_0R)$ is the groupoid of elliptic curves with level-$A$ structure over the commutative ring $\pi_0R$.  

In Section \ref{subsec:elevel}, for derived level-$A$ structures, the assignment $\sx\mapsto\Level(A,\sx/R)$ determines a functor $\El(R)\to\cs$ which classifies a left fibration $\El^A(R)\to\El(R)$ of $\infty$-categories by the unstraightening construction (see \cite[Definition 3.3.2.2 and Section 2.2.1]{lu-HTT}).  Objects of $\El^A(R)$ are triples $(\se,\sd,\phi)$ where $\se$ is a spectral elliptic curve over $R$ and $(\sd,\phi)$ is a derived level-$A$ structure on $\se$ as in Definition \ref{def:elevel}.  

For each $R\in\CAlg^\cn$, consider all spectral elliptic curves over $R$ with level-$A$ structure.  This moduli problem can be thought of as a functor 
\[
    \cm_\El^A\co\CAlg^\cn\to\cs,\quad R\mapsto\El^A(R)^\simeq 
\]
where $\El^A(R)^\simeq$ is the space of spectral elliptic curves $\se/R$ with a derived level-$A$ structure $(\sd,\phi)$.  To prove its representability, we proceed as follows.  

\begin{lemma}
    For each discrete commutative $R_0$, the space $\cm_\El^A(R_0)$ is $1$-truncated.  
\end{lemma}
\begin{proof}
    This follows from the fact that the classical moduli problem above is represented by a Deligne--Mumford $1$-stack.  
\end{proof}

\begin{lemma}
    The functor $\cm_\El^A\co\CAlg^\cn\to\cs$ is an \'etale sheaf.  
\end{lemma}
\begin{proof}
    Let $\{R\to U_i\}$ be an \'etale cover of $R$, and $U_\bullet$ be the associated \v{C}ech-simplicial object.  Consider the diagram 
    \[
        \xymatrix{
            \El^A(R)^\simeq \ar[r]^-f \ar[d]^{p} & \varprojlim_\bDelta\El^A(U_\bullet)^\simeq \ar[d]^{q} \\
            \El(R)^\simeq \ar[r]^-g              & \varprojlim_\bDelta\El(U_\bullet)^\simeq 
	}
    \]
    The map $p$ is a left fibration between Kan complexes, and so is a Kan fibration by \cite[Lemma 2.1.3.3]{lu-HTT}.  The map $q$ is a pointwise Kan fibration.  By picking the projective model structure for the homotopy limit we may assume that $q$ is a Kan fibration as well.  The map $g$ is an equivalence by \cite[Theorem 2.4.1]{lu-EC1}.  To show that $f$ is an equivalence, we need only show that for every $\se\in\El(R)$, the map 
    \[
        p^{-1}(\se)\simeq\Level(A,\se/R)\to\varprojlim_\bDelta\Level(A,\se\times_RU_\bullet/U_\bullet)\simeq q^{-1}g(\se) 
    \]
    is an equivalence.  Observe that $\Level(A,\se/R)$ is a full $\infty$-subcategory of $\CDiv(\se/R)$ and $\varprojlim_\bDelta\Level(A,\se\times_RU_\bullet/U_\bullet)$ is a full $\infty$-subcategory of $\varprojlim_\bDelta\CDiv(\se\times_RU_\bullet/U_\bullet)$.  Since $\CDiv_{\se/R}$ is an \'etale sheaf by Lemma \ref{lem:2}, the functor 
    \[
        \Level(A,\se/R)\to\varprojlim_\bDelta\Level(A,\se\times_RU_\bullet/U_\bullet)
    \]
    is fully faithful.  To show that it is an equivalence, we need only show that it is essentially surjective.  

    Given any $\{(\sd_{U_\bullet},\phi_{U_\bullet})\}$ in $\varprojlim_\bDelta\Level(A,\se\times_RU_\bullet/U_\bullet)$, clearly we can find a morphism $\sd\to\se$ in $\CDiv(\se/R)$ whose image under the equivalence 
    \[
        \CDiv(\se/R)\simeq\varprojlim_\bDelta\CDiv(\se\times_RU_\bullet/U_\bullet) 
    \]
    is $\{\sd_{U_\bullet}\to\se\times_RU_\bullet\}$, along with $\phi\co A\to \se^\heartsuit(\pi_0R)$ lifting $\{\phi_{U_\bullet}\}$.  It remains to show that $(\sd,\phi)$ is a derived level-$A$ structure.  This is true because in the classical case, 
    \[
        \Level(A,\se^\heartsuit/\pi_0R)\simeq\varprojlim_\bDelta\Level(A,\se^\heartsuit\times_{\pi_0R}\pi_0U_\bullet/\pi_0U_\bullet) 
        \qedhere
    \]
\end{proof}

\begin{lemma}
    The functor $\cm_\El^A\co\CAlg^\cn\to\cs$ is nilcomplete, infinitesimally cohesive, and integrable.  
\end{lemma}
\begin{proof}
    Consider the following diagram in $\Fun(\CAlg^\cn,\cs)$: 
    \begin{equation}
    \label{mella}
	\xymatrix{
		\cm_\El^A \ar[r]^-f \ar[rd]_h & \cm_\El \ar[d]^g \\
                                          & \!\,\ast 
	}    
    \end{equation}
    By \cite[Remark 17.3.7.3]{lu-SAG}, since $\cm_\El$ is nilcomplete, infinitesimally cohesive, and integrable from \cite[Theorem 2.4.1]{lu-EC1}, we need only prove that $f$ is so.  By \cite[Proposition 17.3.8.4]{lu-SAG}, $f$ has these properties if and only if each fiber of $f$ does, i.e., for each $R\in\CAlg^\cn$ and a point $\eta_{_\se}\in\cm_\El(R)$ which represents a spectral elliptic curve $\se$, the functor 
    \[
        \CAlg^\cn_R\to\cs,\quad R'\mapsto\cm_\El^A(R')\times_{\cm_\El(R')}\{\eta_{_\se}\} 
    \]
    is nilcomplete, infinitesimally cohesive, and integrable.  This functor is precisely $\Level^A_{\se/R}$, which is so by Theorem \ref{thm:elevel}.  
\end{proof}

\begin{lemma}
    The functor $\cm_\El^A$ admits a cotangent complex which is connective and almost perfect.  
\end{lemma}
\begin{proof}
    Again, let us consider the diagram \eqref{mella}.  Then \cite[Proposition 17.3.9.1]{lu-SAG} reduces us to proving that 
    $f$ admits a cotangent complex.  
    
    By \cite[Proposition 17.2.4.7]{lu-SAG}, a morphism $j\co X\to Y$ in $\Fun(\CAlg^\cn,\cs)$ admits a cotangent complex if, for any corepresentbale $Y'\simeq\Map(R,-)\co\CAlg^\cn\to\cs$ and any natural transformation $Y'\to Y$, $j'$ in the following pullback diagram admits a cotangent complex: 
    \[
	\xymatrix{
            Y'\times_YX \ar[d]^{j'} \ar[r] & X \ar[d]^j \\
		Y' \ar[r]                      & Y 
	}        
    \]
    Thus, to prove that $\cm_\El^A\to\cm_\El$ admits a cotangent complex, we need only prove that, for any $R\in\CAlg^\cn$ and any spectral elliptic curve $\se$ which corresponds to a natural transformation $\Spet R\to\cm_\El$, or to $\eta_{_\se}\in\cm_\El(R)$, the functor 
    \[
        \CAlg^\cn_R\to\cs,\quad R'\mapsto\cm_\El^A(R')\times_{\cm_\El(R')}\{\eta_{_\se}\} 
    \]
    admits a cotangent complex.  Again we identify this functor as $\Level^A_{\se/R}$ and apply Theorem \ref{thm:elevel}.  Moreover, the properties of the desired cotangent complex being connective and almost perfect also follow from those associated with $\Level^A_{\se/R}$.  
\end{proof}

\begin{lemma}
    The functor $\cm_\El^A$ is locally almost of finite presentation.  
\end{lemma}
\begin{proof}
    The morphism $h\co\cm_\El^A\to\ast$ in \eqref{mella} is infinitesimally cohesive and admits an almost perfect cotangent complex.  By \cite[17.4.2.2]{lu-SAG}, it is locally almost of finite presentation.  Therefore $\cm_\El^A$ is locally almost of finite presentation, since $\ast$ is a final object of $\Fun(\CAlg^\cn,\cs)$.  	
\end{proof}

Here is a generalization of \cite[Theorem 2.4.1]{lu-EC1}.  

\begin{theorem}
\label{thm:mella}
    The functor 
    \[
        \cm_\El^A\co\CAlg^\cn\to\cs,\quad R\mapsto\cm_\El^A(R)=\El^A(R)^\simeq 
    \]
    is representable by a spectral Deligne--Mumford $1$-stack which are locally almost of finite presentation over the sphere spectrum.  
\end{theorem}
\begin{proof}
    We apply Theorem \ref{thm:Lurie} and verify the set of conditions one by one through the above series of lemmas.  
\end{proof}

\begin{remark}
\label{rmk:tlevel}
    With motivation from chromatic homotopy theory, as a variant of \cite[Theorem 2.4.1]{lu-EC1}, Lurie proved the representability of a functor of {\em oriented} elliptic curves by a nonconnective spectral Deligne--Mumford stack \cite[Proposition 7.2.10]{lu-EC2} (cf.~\cite[Warning 0.0.5]{lu-EC1}).  In particular, taking global sections of its structure sheaf, he recovered the periodic $\einf$-ring spectrum $\TMF$ of topological modular forms previously constructed by Goerss, Hopkins, and Miller through obstruction theory \cite[Definition 7.0.3]{lu-EC2}.  Given his results, Theorem \ref{thm:mella} is readily adaptable to oriented elliptic curves with level structure and their corresponding periodic spectra of topological modular forms with level structure, though we do not include the details here.  The interested reader may compare Theorems \ref{thm:dlevel} and \ref{thm:jlr} below, including their proofs.  In view of \cite{hill2016topological}, it would be interesting to discuss suitably compactified objects in this setting as well (cf.~\cite{davies2025tate}).  
\end{remark}

\subsection{Higher-homotopical Lubin--Tate towers}
\label{subsec:LT}

Based on Section \ref{subsec:plevel}, specifically \ref{subsubsec:pdiv}, our goal in this subsection is to generalize to higher levels some of the representability results on {\em deformations} of $p$-divisible groups from \cite[\S3]{lu-EC2}.  

Let $\bG_0$ be a nonstationary $p$-divisible group over a commutative ring $R_0$ of height $h$, and $R\in\CAlg^\ad_\cpl$.  Recall from \cite[Definitions 3.1.4 and 3.1.1]{lu-EC2} that a {\em deformation of $\bG_0$ over $R$} is a spectral $p$-divisible group $\bG$ over $R$ together with an equivalence class of $\bG_0$-taggings of $\bG$.  A main result therein is a representability theorem with the {\em spectral deformation ring} $R^\un_{\bG_0}$ representing the moduli problem of such deformations \cite[Theorem 3.0.11]{lu-EC2} (see also \cite[Theorems 3.1.15 and 3.4.1]{lu-EC2}).  

\begin{definition}
\label{def:dlevel}
    As in Section \ref{subsubsec:pdiv}, we let $\Level(r,\bG/R)$ denote the space of derived level-$(\bz/p^r\bz)^h$ structures on $\bG$.  Consider the functor 
    \[
        \Def_{\bG_0}^r\co\CAlg^\ad_\cpl\to\cs,\quad R\mapsto\Def(r,\bG_0,R) 
    \]
    Here $\Def(r,\bG_0,R)$ denotes the $\infty$-category whose objects are triples $(\bG,\alpha,\lambda)$ such that 
    \begin{itemize}
        \item $\bG$ is a spectral $p$-divisible group over $R$, 
    
        \item $\alpha$ is an equivalence class of $\bG_0$-taggings of $\bG$, and 

        \item $\lambda\in\Level(r,\bG/R)$ as in Definition \ref{def:plevel}.  
    \end{itemize}
    We call each object $(\bG,\alpha,\lambda)$ a {\em deformation of $\bG_0$ over $R$ with level-$(\bz/p^r\bz)^h$ structure}.  
\end{definition}

\begin{remark}
    Given $\alpha$ as above, a derived level-$(\bz/p^r\bz)^h$ structure $\lambda$ on $\bG/R$ determines a level-$(\bz/p^r\bz)^h$ structure on $\bG_0/R_0$ (up to an extension of scalars) by base change along $\alpha$.  
\end{remark}

\begin{theorem}
\label{thm:dlevel}
    With the above hypotheses and notations, the functor $\Def_{\bG_0}^r$ is corepresentable by an $\einf$-ring whose 0'th homotopy group is finite over that of the spectral deformation ring $R^\un_{\bG_0}$.  
\end{theorem}
\begin{proof}
    Following \cite[4.3.4]{katz1985arithmetic}, we view this moduli problem as a product of a representable one and a relatively representable one, and apply \cite[Theorem 3.1.15]{lu-EC2} and Theorem \ref{thm:plevel} consecutively.  
    
    To be specific, let $\bG_\univ/R^\un_{\bG_0}$ denote the universal deformation of $\bG_0/R_0$ from \cite[Theorem 3.1.15]{lu-EC2}.  Suppose that $\bG$ is a spectral deformation of $\bG_0$ to $R$ classified by a map of $\einf$-rings $R^\un_{\bG_0}\to R$, along which $\bG\simeq\bG_\univ\times_{R^\un_{\bG_0}}R$ as spectral $p$-divisible groups over $R$.  We then obtain from Theorem \ref{thm:plevel} and Remark \ref{rmk:nonconnective} 
    \[
        \Level(r,\bG/R)\simeq\Level(r,\bG_\univ\times_{R^\un_{\bG_0}}R/R)\simeq\Map_{\CAlg_{R^\un_{\bG_0}}}(\cp^r_{\bG_\univ/R^\un_{\bG_0}},R) 
    \]
    where $\cp^r_{\bG_\univ/R^\un_{\bG_0}}$ classifies derived level-$(\bz/p^r\bz)^h$ structures on the universal deformation $\bG_\univ/R^\un_{\bG_0}$ as a spectral $p$-divisible group, with $\pi_0\cp^r_{\bG_\univ/R^\un_{\bG_0}}$ finite over $\pi_0R^\un_{\bG_0}$ by affineness (cf.~\cite[Remark 5.2.0.2]{lu-SAG}).  

    Let us verify that the $\einf$-ring $\cp^r_{\bG_\univ/R^\un_{\bG_0}}$ is as desired.  Indeed, consider the functor 
    \[
        \CAlg^\ad_\cpl\to\cs,\quad R\mapsto\Map_{\CAlg^\ad_\cpl}(\cp^r_{\bG_\univ/R^\un_{\bG_0}},R) 
    \]
    Given $R\in\CAlg^\ad_\cpl$, $\Map_{\CAlg^\ad_{\cpl,R_0}}(\cp^r_{\bG_\univ/R^\un_{\bG_0}},R)$ can be viewed as the $\infty$-category of pairs $(f,g)$ where 
    \[
        f\co R^\un_{\bG_0}\to R 
    \]
    along the structure morphism of $\Spet\cp^r_{\bG_\univ/R^\un_{\bG_0}}$ over $R^\un_{\bG_0}$ classifies a deformation $(\bG,\alpha)$ over $R$ of $\bG_0/R_0$, and 
    \[
        g\in\Map_{\CAlg_{R^\un_{\bG_0}}}(\cp^r_{\bG_\univ/R^\un_{\bG_0}},R)\simeq\Level(r,\bG/R)
    \]
    along the restriction from $\CAlg^\ad_\cpl$ to $\CAlg$ (cf.~claim (ii) in \cite[proof of Theorem 3.4.1]{lu-EC2}) specifies a derived level-$(\bz/p^r\bz)^h$ structure $\lambda$ on $\bG/R$.  Thus we recover precisely the functor $\Def_{\bG_0}^r$ as in Definition \ref{def:dlevel}.  
\end{proof}

Although we have obtained $\einf$-rings as classifying objects associated with the above spectral moduli problems, these spectra may be complicated, and we do not know yet their homotopy groups in general (see Section \ref{subsubsec:hhg} below).  In algebraic topology, the orientation of an $\einf$-ring spectrum makes the $E_2$-page of its associated Atiyah--Hirzebruch spectral sequence degenerate and gives us certain information of its homotopy groups.  

\begin{definition}
\label{def:ordlevel}
    Let $\bG_0$ be a height-$h$ $p$-divisible group over $R_0$ as above.  Define the functor 
    \[
        \Def_{\bG_0}^{\Or,r}\co\CAlg^\ad_\cpl\to\cs,\quad R\mapsto\Def^\Or(r,\bG_0,R) 
    \]
    where $\Def^\Or(r,\bG_0,R)$ is the space of {\em oriented deformation $(\bG,\alpha,e,\lambda)$ of $\bG_0$ over $R$ with level-$(\bz/p^r\bz)^h$ structure}, with 
    \begin{itemize}
        \item $\bG$ a spectral $p$-divisible group over $R$, 
        
        \item $\alpha$ an equivalence class of $\bG_0$-taggings of $\bG$, 
        
        \item $e\co S^2\to\Omega^\infty\bG^\circ(\tau_{\geq0}R)$ an orientation of the identity component of $\bG$, and 
        
        \item $\lambda$ a derived level-$(\bz/p^r\bz)^h$ structure on $\bG$.  
    \end{itemize}
\end{definition}

\begin{theorem}
\label{thm:jlr}
    Hypotheses and notations as above, for each nonnegative integer $r$, the functor $\Def_{\bG_0}^{\Or,r}$ is corepresentable by an $\einf$-ring, depending functorially on $\bG_0/R_0$, as an algebra over the oriented deformation ring $R^\Or_{\bG_0}$.  Moreover, its 0'th homotopy group is finite over $\pi_0R^\Or_{\bG_0}$.  
\end{theorem}
\begin{proof}
    Let $\Def^\Or(\bG_0,R)$ denote the $\infty$-groupoid of triples $(\bG,\alpha,e)$ where $\bG$ is a $p$-divisible group of over $R$, $\alpha$ is an equivalence class of $\bG_0$-taggings of $\bG$, and $e$ is an orientation of the identity component of $\bG$.  By \cite[Remark 6.0.7]{lu-EC2}, the functor 
    \[
        \Def_{\bG_0}^\Or\co\CAlg^\ad_\cpl\to\cs,\quad R\mapsto\Def^\Or(\bG_0,R) 
    \]
    is corepresented by $R^\Or_{\bG_0}$, i.e., there is an equivalence of spaces $\Map_{\CAlg^\ad_\cpl}(R^\Or_{\bG_0},R)\simeq\Def^\Or(\bG_0,R)$.  Let $\bG^\Or_\univ$ be the associated universal oriented deformation of $\bG_0$ over $R^\Or_{\bG_0}$.  Then, analogous to the unoriented case discussed above, the $\einf$-ring $\cp^r_{\bG^\Or_\univ/R^\Or_{\bG_0}}$ from Theorem \ref{thm:plevel} is the desired spectrum.  
\end{proof}

\begin{remark}
\label{rmk:LT}
    We shall call this $\einf$-ring from Theorem \ref{thm:jlr} a {\em Jacquet--Langlands spectrum of level $p^r$} and denote it by $\jl_r$.  Indeed, by functoriality, $\jl_r$ admits an action of $\GL_h(\bz/p^r\bz)\times\Aut(\bG_0)$.  Moreover, since $(\bz/p^r\bz)^h\subset(\bz/p^{r+1}\bz)^h$, we obtain a tower of spectral Deligne--Mumford moduli stacks 
    \[
        \begin{tikzpicture}
            \node (L7) at (0, 5.54) {$\vdots$};
            \node (L6) at (0, 5.25) {};
            \node (L5) at (0, 4.35) {$\Spet\jl_{r+1}$};
            \node (L4) at (0, 3.3) {$\Spet\jl_r$};
            \node (L3) at (0, 2.4) {};
            \node (L2) at (0, 2.3) {$\vdots$};
            \node (L1) at (0, 2) {};
            \node (L0) at (0, 1.1) {$\Spet\jl_0$};
            \draw [-stealth] (L6) -- (L5);
            \draw [-stealth] (L5) -- (L4);
            \draw [-stealth] (L4) -- (L3);
            \draw [-stealth] (L1) -- (L0);
        \end{tikzpicture}
    \]
    Unwinding the definitions, we see that by construction its 0'th homotopy recovers the finite levels of the Lubin--Tate tower of $\bG_0/R_0$ (cf.~\cite[Lemma 7.1]{galatius2018derived}).  
    In classical arithmetic algebraic geometry, the Lubin--Tate tower can be used to realize the Jacquet--Langlands correspondence \cite[Theorem B and Chapter II]{harris2001geometry}.  Naturally we may ask if there is a topological (or higher-homotopical) realization of this correspondence.  Such a construction has appeared recently \cite{Salch2023elladicTJ}.  In contrast to our approach, the methods therein are based on the Goerss--Hopkins--Miller--Lurie sheaf, by considering certain degenerate level structures whose representing objects are {\em \'etale} over the Lubin--Tate space carrying universal deformations.  {\em Integrally}, it would be interesting if our higher-categorical analogues of Lubin--Tate towers above can also lead to a topological version of the classical Jacquet--Langlands correspondence, which means that we construct representations in the category of spectra (see further Section \ref{subsubsec:rep} below).  
\end{remark}

\subsection{Topological lifts of power operation rings}
\label{subsec:po}

Continuing Section \ref{subsubsec:nonfull}, here we consider certain non-full level structures relevant to power operations in Morava E-theories.  Let us first recall the classical deformation theory of one-dimensional commutative formal groups.  

Given a formal group $\hG_0$ over a perfect field $k$ of characteristic $p$, a deformation of $\hG_0$ over a complete local ring $R$ is a triple $(\hG,i,\eta)$ such that 
\begin{itemize}
    \item $\hG$ is a formal group over $R$, 

    \item $i\co k\to R/\fm$ is a ring homomorphism, with $\fm$ the maximal ideal of $R$, and 

    \item $\eta\co\pi^*\hG\simeq i^*\hG_0$ is an isomorphism of formal groups over $R/\fm$, with $\pi\co R\to R/\fm$ the natural projection.  
\end{itemize}
We simply write $\hG$ for a deformation if $(i,\eta)$ is understood.  

Recall the relative Frobenius isogeny $\Frob\co\hG_0\to\sigma^*\hG_0$ over $k$, with $\sigma\co k\to k$, $x\mapsto x^p$.  For each nonnegative integer $r$, a deformation of the $p^r$-power Frobenius $\Frob^r$ over $R$ consists of deformations $(\hG,i,\eta)$ and $(\hG',i',\eta')$ of $\hG_0$ over $R$, together with an isogeny $\psi\co\hG\to\hG'$ of formal groups over $R$, such that the following compatibility conditions hold (cf.~\cite[\S5.5]{zhu2020norm}): 
\begin{enumerate}
    \item The triangle 
    \[
        \xymatrix{
            k \ar[r]^-{i'} \ar[d]_{\sigma^r} & R/\fm \\
		k \ar[ur]^i                      & 
	}
    \]
    commutes, so that $(i')^*\hG_0=i^*(\sigma^r)^*\hG_0$.  

    \item The rectangle 
    \[
        \xymatrix{
		\pi^*\hG \ar[rr]^{\pi^*(\psi)} \ar[d]_\eta && \pi^*\hG' \ar[d]^{\eta'} \\
		i^*\hG_0 \ar[rr]^-{i^*(\Frob^r)}           && i^*(\sigma^r)^*\hG_0 
	}
    \]
    of formal groups over $R/\fm$ commutes.  
\end{enumerate}

In particular, when $\pi^*(\psi)$ equals the $p^r$-power relative Frobenius isogeny on $\pi^*\hG$, the association of $\psi$ with $\hG$ is equivalent to the choice of a subgroup scheme $\hH=\ker(\psi)$ of $\hG$ that is {\em cyclic of order $p^r$} in the sense of \cite[6.1]{katz1985arithmetic} (cf.~the $\Level_0(r,G/R)$-structure in Definition \ref{def:nonfull}).  

We say that two deformations of Frobenius 
\[
    (\hG_1,i_1,\eta_1)\to(\hG_1',i_1',\eta_1')\quad\text{and}\quad(\hG_2,i_2,\eta_2)\to(\hG_2',i_2',\eta_2') 
\]
are isomorphic, if their sources are isomorphic and their target are isomorphic, both as deformations of $\hG_0$ over $R$, i.e., $\psi\co\hG_1\simeq\hG_2$, $i_1=i_2$, $\eta_1=\eta_2\circ\pi^*(\psi)$, and $\psi'\co\hG_1'\simeq\hG_2'$, $i_1'=i_2'$, $\eta_1'=\eta_2'\circ\pi^*(\psi')$.  

We have the following classification theorem for deformations of Frobenius from \cite[Theorem 42 and Section 13]{strickland1997finite}, as reformulated in \cite[Proposition 5.7 and Corollary 5.12]{zhu2020norm} (cf.~\cite[Proposition 3.5]{rezk2013power}).  

\begin{theorem}[Strickland]
\label{thm:Strickland}
    Let $\hG_0$ be a height-$h$ formal group over a perfect field $k$ of characteristic $p$.  For each nonnegative integer $r$, there exists a complete local ring $A_r$ which carries a universal deformation of the $p^r$-power Frobenius 
    \[
        \psi^r_\univ\co(\hG^r_s,i^r_s,\eta^r_s)\to(\hG^r_t,i^r_t,\eta^r_t) 
    \]
    Namely, given complete local rings $R$ and local homomorphisms $f\co A_r\to R$, the assignment $\psi^r_\univ\mapsto f^*(\psi^r_\univ)$ gives a bijection between the set of local homomorphisms $f\co A_r\to R$ and the set of isomorphism classes of deformations of the $p^r$-power Frobenius over $R$.  Moreover, the following hold: 
    \begin{enumerate}
        \item When $r=0$, the ring $A_0$ is the Lubin--Tate deformation ring of $\hG_0/k$, isomorphic to $W(k)\llbracket v_1,\ldots,v_{h-1}\rrbracket$.  

        \item There is a local homomorphism $s^r\co A_0\to A_r$ classifying the source of $\psi^r_\univ$ such that 
        \[
            (\hG^r_s,i^r_s,\eta^r_s)=\big((s^r)^*\hG_\univ,\id_k,\id_{\hG_0}\big) 
        \]
        with $(\hG_\univ,\id_k,\id_{\hG_0})$ over $A_0$ the universal deformation of $\hG_0$, along which $A_r$ is finite and free as an $A_0$-module.  

        \item There is a local homomorphism $t^r\co A_0\to A_r$ classifying the target of $\psi^r_\univ$ such that 
        \[
            (\hG^r_t,i^r_t,\eta^r_t)=\big((t^r)^*\hG_\univ,\sigma^r,\id_{(\sigma^r)^*\hG_0}\big) 
        \]
    \end{enumerate}
\end{theorem}

These rings $A_r$ also bear topological meanings in relationship to the Morava E-theory $E$ of $\hG_0/k$, as in \cite[Theorem 1.1]{strickland1998morava}.  

\begin{theorem}[Strickland]
    Hypotheses and notations as above, there is a natural ring isomorphism 
    \[
        A_r\simeq E^0(B\Sigma_{p^r})/I_\tr 
    \]
    where $I_\tr$ is the ideal generated by the images of transfers from proper subgroups of $\Sigma_{p^r}$.  
\end{theorem}

\begin{remark}
    Strickland's proof relies on rational {\em computations} which match up the ranks of the two sides of this isomorphism as $E^0$-modules (see \cite[Theorems 9.2 and 8.6]{strickland1998morava}).  From the perspective of \cite[Example 0.0.6 and Theorem 0.0.8]{lu-EC2}, this theorem can be viewed as a ``partial'' realization by $E$-cohomology of certain spaces modulo equivalences, in the setting of $\einf$-ring spectra, of the solution to a moduli problem in classical deformation theory.  
\end{remark}

The collection $\{A_r\}_{r\geq0}$ has the structure of a graded coalgebra over $A_0$, with structure maps 
\[
    s=s^r\co A_0\to A_r,\quad t=t^r\co A_0\to A_r,\quad\mu^{m,n}\co A_{m+n}\to A_m{^s\otimes^t_{A_0}}A_n 
\]
which classify the source, target, and composite of deformations of Frobenius, respectively.  In particular, given a $K(h)$-local $E$-algebra $F$, there is a $p^r$-power operation $F^0(X)\to F^0(X\times B\Sigma_{p^r})/I_\tr$, and when $X=*$, we have 
\[
    \pi_0F\to E^0(B\Sigma_{p^r})/I_\tr\otimes_{E^0}\pi_0F\simeq A_r{^s\otimes_{A_0}}\pi_0F 
\]
These equip $\pi_0F$ with the structure of a $\Gamma$-module, where the $\bz_{\geq0}$-graded pieces of $\Gamma$ are $A_0$-linear duals of $A_r$ along $s^r$.  For more details about power operations in Morava E-theories, the interested reader may refer to \cite{rezklecture2006,rezk2009congruence,rezk2013power}.  Explicit computations of $\Gamma$ have been carried out at height 2 in \cite{rezk2008power} for the prime 2, \cite{zhu2014power} for the prime 3, and \cite{zhu2019semistable} for all primes, through relevant moduli schemes of level structures.  The cases at height greater than 2 are still lack of quantitative understanding.  

As observed above, the assignment $\psi\mapsto\ker(\psi)$ gives a one-to-one correspondence between isomorphism classes of deformations of $\Frob^r$ with source $\hG$ and cyclic degree-$p^r$ finite flat subgroup scheme of $\hG$.  Therefore, we see that $A_r$ corepresents the moduli problem 
\[
    (\CAlg^\ad_\cpl)^\heartsuit\to\Set,\quad R\mapsto\Def_0(r,\hG_0,R) 
\]
where $\Def_0(r,\hG_0,R)$ consists of pairs $\hH\subset\hG$, with $\hG$ a deformation of $\hG_0$ over $R$ and $\hH$ a cyclic subgroup of order $p^r$.  

\begin{theorem}
\label{thm:ehr}
    Hypothesis and notations as in Theorem \ref{thm:Strickland}, for each positive integer $r$, there exists an $\einf$-ring spectrum $E_{h,r}$ such that $\pi_0(E_{h,r})\simeq A_r$, which depends functorially on $\hG_0/k$.  
\end{theorem}
\begin{proof}
    Given the formal group $\hG_0$ over the perfect field $k$ of characteristic $p$, necessarily nonstationary \cite[Example 3.0.10]{lu-EC2}, we view it as a connected $p$-divisible group and consider instead the functor 
    \[
        \CAlg^\ad_\cpl\to\cs,\quad R\mapsto(\bG,\alpha,e,\lambda) 
    \]
    where $(\bG,\alpha)$ is a {\em spectral} deformation of $\hG_0$ over $R$ as in Section \ref{subsec:LT}, $e$ is an orientation of $\bG^\circ$, and $\lambda\in\Level_0(r,\bG/R)$ is a derived level structure as in Section \ref{subsubsec:nonfull}, esp.~Remark \ref{rmk:pdiv}.  This functor lifts the one corepresented by $A_r$ precisely as the case of $r=0$ \cite[Remarks 6.4.8 and 3.0.14]{lu-EC2}.  Analogous to the proof of Theorem \ref{thm:jlr} (or rather the more elaborated one for Theorem \ref{thm:dlevel}) with full level structures, we then deduce the current theorem from the representability of the spectral moduli problem $\Level^{0,r}_{\bG^\Or_\univ/R^\Or_{\hG_0}}$ by an affine spectral Deligne--Mumford stack.  
\end{proof}

\begin{remark}
\label{rmk:nonlocal}
    Neither the $\einf$-ring spectra $E_{h,r}$ from above nor those from Theorem \ref{thm:jlr} with $r>0$ are $K(h)$-local or complex oriented, in light of \cite[Theorems 1.3 and 3.5]{devalapurkar2020roots} and \cite[Construction 5.1.1 and Remark 5.1.2]{lu-EC2} (cf.~\cite[Section 1.3]{Salch2023elladicTJ} and \cite[Proposition 2]{SVW}).  In particular, they are not {\em finite} algebras over the Morava E-theory spectrum $E_h$, even though the morphisms $\Spet E_{h,r}\to\Spet E_h$ of spectral Deligne--Mumford stacks {\em are} finite (cf.~\cite[Definition 7.2.2.1]{lu-HA}, \cite[Proposition 2.7.2.1]{lu-SAG}, and the finiteness conditions we recalled at the beginning of Section \ref{subsec:isog}, esp.~Remark \ref{rmk:finite}).  The authors thank Sanath Devalapurkar and Akira Tominaga for bringing this point to their attention.  
\end{remark}

\begin{remark}
\label{rmk:pile}
    Although we obtained spectra whose 0'th homotopy groups recover the power operation rings of Morava E-theories, we do not know yet the higher homotopy groups of these spectra concretely or explicitly, as these spectra are not 2-periodic in general unless $r=0$ (or $h=1$, in this case all $E_{h,r}$ are canonically isomorphic to $E_h$), and as they are not \'etale over E-theory spectra.  This non-2-periodicity with $r>0$ should be a manifestation of the structure of a {\em pile}, i.e., a presheaf of categories (rather than of groupoids), as indicated in \cite[Section 4.3]{rezk2014isogenies} (cf.~\cite[Question 1.3]{hill2016topological}).  See some further discussion in Section \ref{subsubsec:hhg} below.  
\end{remark}

\section{More applications}
\label{sec:dual}

\subsection{Jacquet--Langlands spectra}
\label{subsec:JL}

The Langlands program has arguably developed into a paradigm in contemporary mathematics and related fields which connects multiple subareas, including number theory, representation theory, and harmonic analysis, in a precise way.  To be more specific, the global Langlands correspondence gives a (partly conjectural) bijection between 
\begin{itemize}
    \item $n$-dimensional complex linear representations of the absolute Galois group $\Gal_F$ of a given number field $F$ and 

    \item automorphic representations of the general linear group $\GL_n(\ba_F)$ with coefficients in the ring of ad\`eles of $F$ that arise within the representations given by functions on the double coset space $\GL_n(F)\backslash\GL_n(\ba_F)/\GL_n(\cO)$, with $\cO=\prod_\nu\cO_\nu$ the product of the rings of integers of completions at all valuations of $F$ 
\end{itemize}
which is compatible with certain $L$-function conditions.  More generally, the group $\GL_n$ may be replaced by any reductive group.  The Langlands correspondence has many specific examples in number theory.  For the group $\GL_1$, this correspondence recovers global class field theory.  For $\GL_2$ it affords the famous modularity theorem for semistable elliptic curves \cite{wiles1995modular,taylor1995ring}.  

The Langlands correspondence has a local version.  Let $E$ be a local field and $G$ be a reductive group over $E$.  The local Langlands correspondence predicts that given any irreducible smooth representation $\pi$ of $G(E)$, we can naturally associate an $L$-parameter, i.e., a continuous homomorphism $\phi_\pi\co W_E\to\ll G(\bc)$, where $W_E$ is the Weil group of $E$, and $\ll G$ is the Langlands dual group of $G$.  

Of particular relevance to this paper is the Jacquet--Langlands correspondence.  Let $K$ be a $p$-adic field and $D$ be a central division algebra over $K$ of dimension $d^2$.  We fix a positive integer $r$ and let $G=\GL_n(K)$, $G'=\GL_r(D)$, where $n=rd$.  The Jacquet--Langlands correspondence aims to relate irreducible smooth representations of $G$ to those of $G'$, while the local Langlands correspondence relates such representations of $G$ to $n$-dimensional complex representations of $W_K$.  

We shall focus on the case of $r=1$, when $D$ is a central division algebra over $K$ of dimension $n^2$ (and invariant $1/n$).  There is a Jacquet--Langlands bijection between square integrable representations of $\GL_n(K)$ and such representations of $D^\times$.  In Section \ref{subsec:LT}, we have built a higher-homotopical realization of the Lubin--Tate tower associated to a nonstationary $p$-divisible group over a commutative ring, and discussed its relevance to a potential topological version of the Jacquet--Langlands correspondence (see Remark \ref{rmk:LT}).  On the other hand, we know the actions of certain Galois groups and automorphism groups on certain objects, such as Morava E-theories, topological Hochschild homology, and topological cyclic homology.  In particular, these groups act on their homotopy groups.  For example, we have the action of Morava stabilizer groups $\bg_h$ on Morava E-theories $E_h$.  It can be used to compute homotopy groups of the ($K(h)$-local) sphere spectrum via a spectral sequence 
\begin{equation}
\label{DH}
    E_2^{s,t}=H^s_\cts(\bg_h,\pi_tE_h)\implies\pi_{t-s}\bs_{K(h)} 
\end{equation}
where the $E_2$-page consists of continuous cohomology groups of $\bg_h$ \cite{devinatz2004homotopy}.  In general, however, this group cohomology is complicated to compute, which manifests a common phenomenon in the Langlands program, i.e., the Galois side is usually harder to understand than the automorphic side.  One strategy for relevant problems is to transfer from the Galois side to the automorphic side.  Let us see an example first \cite{barthel2024rationalization}.  

\begin{theorem}[Barthel--Schlank--Stapleton--Weinstein]
    There is an isomorphism of graded $\bq$-algebras 
    \[
        \bq\otimes\pi_*\bs_{K(h)}\simeq\Lambda_{\bq_p}(\zeta_1,\zeta_2,\ldots\zeta_h) 
    \]
    where the latter is the exterior $\bq_p$-algebra with generators $\zeta_i$ in degree $1-2i$.  
\end{theorem}

A key step in the proof of this theorem is to leverage the equivariantly isomorphic Lubin--Tate tower and Drinfeld tower \cite{fargues2008isomorphisme,scholze2013moduli} between the generic fibers, {\em rationally} transferring the computation of cohomology of $\bg_h$ to that for the Drinfeld symmetric space $\ch\simeq\Dr_K$ \cite[Theorem 3.9.1]{barthel2024rationalization}: 
\begin{equation}
\label{LTD}
    \begin{tikzpicture}[baseline={([yshift=-10pt]current bounding box.north)}]
        \node (L) at (0,0) {$\LT_{\!K}$};
        \node (R) at (4,0) {$\Dr_K$};
        \node (T) at (2,3) {$\cm_\infty^\lt\simeq\cm_\infty^\dr$};
        \draw [-stealth] (T) to [out=180,in=90] node [left] {$\scriptstyle\GL_h(\cO_K)$} (L);
        \draw [-stealth] (T) to [out=0,in=90] node [right] {$\scriptstyle\cO_D^\times\,\simeq\,\bg_h$} (R);
    \end{tikzpicture}
\end{equation}
In a sequel \cite{barthel2024hopkinspicardgroup}, Barthel, Schlank, Stapleton, and Weinstein compute the Picard group of $K(h)$-local spectra by using some results on the Drinfeld symmetric space from \cite{colmez2020cohomology,colmez2021integral}.  

At the ground level of the Lubin--Tate tower, we know by work of Goerss, Hopkins, Miller, and Lurie that the Lubin--Tate space has a higher-homotopical refinement, namely, the Morava E-theory spectrum of the associated formal group over the residue field $k$ of $K$, as an $\einf$-ring (or spectral Deligne--Mumford stack).  It is thus a natural question how to lift the two towers to the higher-categorical setting of spectral algebraic geometry, even {\em integrally}, as a fully structured apparatus affording transfers such as those above and more.  

Let $\bG_0$ be a height-$h$ $p$-divisible group over $\bar{k}$.  In Section \ref{subsec:LT}, for each nonnegative integer $r$, we considered the functor 
\[
    \Def_{\bG_0}^{\Or,r}\co\CAlg^\ad_\cpl\to\cs,\quad R\mapsto\Def^\Or(r,\bG_0,R) 
\]
and showed that it is corepresentable by an $\einf$-ring $\jl_r$.  The duality of Lubin--Tate and Drinfeld towers then leads us to the following.  

\begin{definition}
\label{def:JL}
    Define the {\em Jacquet--Langlands spectrum} $\jl$ to be the limit of the spectra $\jl_r$ from Theorem \ref{thm:jlr}, i.e., $\jl\coloneqq\varprojlim_r\jl_r$.  
\end{definition}

\begin{proposition}
\label{prop:JL}
    The spectrum $\jl$ is an $\einf$-ring.  
\end{proposition} 
\begin{proof}
    This is because the $\infty$-category of $\einf$-rings admits inverse limits.  See \cite[Corollary 3.2.2.4]{lu-HA} for details.  
\end{proof}

\begin{remark}
    This spectrum is a higher-homotopical realization of $\cm_\infty^\lt\simeq\cm_\infty^\dr$ in \eqref{LTD}, the Lubin--Tate/Drinfeld moduli space of height $h$ for $K$ at infinite level, which has the structure of a perfectoid space \cite{scholze2013moduli}.  
\end{remark}

\subsection{Jacquet--Langlands duals of Morava E-theory spectra}

By definitions of the Jacquet--Langlands spectra from Theorem \ref{thm:jlr} and Definition \ref{def:JL}, each $\jl_r$ admits an action by $\GL_h(\bz/p^r\bz)\times\bg_h$, and thus $\jl$ admits an action by $\varprojlim_r\GL_h(\bz/p^r\bz)\times\bg_h\simeq\GL_h(\bz_p)\times\bg_h$.  Descending along the Drinfeld tower in \eqref{LTD} yields the following.  

\begin{definition}
\label{def:dual}
    Given a height-$h$ formal group $\hG_0$ over a perfect field of characteristic $p$, let $E_h=\jl_0$ be the associated Morava E-theory spectrum, and $\jl$ be the associated Jacquet--Langlands spectrum at infinite level.  Define the {\em Jacquet--Langlands dual of $E_h$} to be the homotopy fixed point spectrum $\jl^{\h\bg_h}$, denoted by $\ll E_h$.  
\end{definition}

\begin{remark}
    The generic fiber of $\pi_0\ll E_h$ recovers the Drinfeld symmetric space \cite{drinfel1976coverings}.  It is the rigid-analytic space $\ch=\bp^{h-1}_K\setminus\bigcup H$, where $\bp_K^{h-1}$ is a rigid-analytic projective space, and $H$ runs over all $K$-rational hyperplanes in $\bp_K^{h-1}$.  It is the generic fiber of a formal scheme which parametrizes deformations of a certain formal $\cO_D$-module related to $\hG_0$.  In future work, we aim to show that $\ll E_h$ arises from a corresponding spectral moduli problem.  
\end{remark}

In \cite{rognes2008galois}, Rognes defines Galois extensions for commutative ring spectra.  Suppose that $E$ is an $\einf$-ring spectra and $F$ is an $\einf$ $E$-algebra with an action of a finite group $G$.  We say that $F$ is a {\em $G$-Galois extension of $E$}, if $F^{\h G}\simeq E$ and $F\otimes_EF\to\prod_GF$ is an equivalence.  

\begin{proposition}
\label{prop:Galois}
    Hypotheses and notations as above, each $\jl_r$ is a $\GL_h(\bz/p^r\bz)$-Galois extension of $R^\Or_{\hG_0}$.  
\end{proposition}
\begin{proof}
    We fix $r$ and let $d\coloneqq\#\GL_h(\bz/p^r\bz)$.  By functoriality of the moduli problems, $\jl_r^{\h\GL_h(\bz/p^r\bz)}$ is equivalent to $R^\Or_{\hG_0}$.  We need only show that 
    \[
        \jl_r\otimes_{R^\Or_{\hG_0}}\jl_r\xrightarrow{\simeq}\prod_{\GL_h(\bz/p^r\bz)}\jl_r 
    \]
    By \cite[Proposition 1.4.11.1]{lu-SAG}, this is equivalent to 
    \[
        \Spet\jl_r\times_{\Spet R^\Or_{\hG_0}}\Spet\jl_r\xleftarrow{\simeq}\coprod_{\GL_h(\bz/p^r\bz)}\Spet\jl_r\hskip1.35cm 
    \]
    for which we resort to modular interpretations.  Given a connective  $\einf$-ring $R$, let us consider the diagram 
    \[
        \xymatrix{
            \underset{\GL_h(\bz/p^r\bz)}{\coprod}\Spet\jl_r \ar[rd] & & \\
            & \Spet\jl_r\times_{\Spet R^\Or_{\hG_0}}\Spet\jl_r \ar[r] \ar[d] & \Spet\jl_r \ar[d] \\
            & \Spet\jl_r \ar[r] & \Spet R^\Or_{\hG_0} 
        } 
    \]
    where: 
    \begin{itemize}
        \item The moduli space $\Spet\jl_r\times_{\Spet R^\Or_{\hG_0}}\Spet\jl_r$ parametrizes 
        \[
            \{(\bG,\alpha,e,\lambda),(\bG',\alpha',e',\lambda')\} 
        \]
        where each of the two quadruples is an oriented deformation of $\hG_0$ over $R$ with level-$(\bz/p^r\bz)^h$ structure as in Definition \ref{def:ordlevel}.  Since the fiber product is over $\Spet R^\Or_{\hG_0}$, we have $\bG=\bG'$, $\alpha=\alpha'$, and $e=e'$.  Thus this moduli space carries a universal example 
        \[
            (\bG_\univ,\alpha_\univ,e_\univ,\lambda_\univ,\lambda'_\univ) 
        \]

        \item Over $\Spet R^\Or_{\hG_0}$, the moduli space $\underset{\GL_h(\bz/p^r\bz)}{\coprod}\Spet\jl_r$ parametrizes 
        \[
            \left\{\big(\bG,\alpha,e,g(\lambda)\big)\right\}_{g\in\GL_h(\bz/p^r\bz)} 
        \]
        where $(\bG,\alpha,e)$ is an oriented deformation of $\hG_0$ over $R$, $\lambda$ is a level-$(\bz/p^r\bz)^h$ structure, and each $g$ acts on $\lambda$ by a change of Drinfeld basis (cf.~Definition \ref{def:plevel}).  This moduli space in turn carries a universal example 
        \[
            \big(\bG_\univ,\alpha_\univ,e_\univ,\lambda_\univ,g_\univ(\lambda_\univ)\big) 
        \]
    \end{itemize}
    Since $\GL_h(\bz/p^r\bz)$ acts freely and transitively on $(\bz/p^r\bz)^h$, we see that these two moduli are equivalent.  
\end{proof}

\begin{theorem}
\label{thm:dual}
    Given a Morava E-theory spectrum $E_h$ associated to a formal group $\hG_0$ over a perfect field $k$ of characteristic $p$, its Jacquet--Langlands dual spectrum $\ll E_h$ is an $\einf$-ring.  
\end{theorem}
\begin{proof}
    By construction, we have $\ll E_h=\jl^{\h\bg_h}\simeq\varprojlim_r\jl_r^{\h\bg_h}$.  Thus it suffices to show that each $\jl_r^{\h\bg_h}$ is an $\einf$-ring spectrum.  By \cite[Theorem 5.4.4\,(d)]{rognes2008galois} and Proposition \ref{prop:Galois}, we have Galois extensions 
    \[
        \bs_{K(h)}\to E_h\to\jl_r 
    \]
    In view of this, we notice that $\jl_r$ is a profinite $\GL_h(\bz/p^r\bz)\times\bg_h$-spectrum in the sense of \cite[Definition 5.1]{quick2013profinite} (cf.~\cite{davis2016profinite,behrens2010galois}).  In particular, it is a profinite $\bg_h$-spectrum.  By \cite[Proposition 3.23]{quick2013continuous},  we then have 
    \[
        \jl_r^{\h\bg_h}\simeq\Tot\big(\Map(\bg_h^\bullet,\jl_r)\big) 
    \]
    Again, as the $\infty$-category $\CAlg$ admits inverse limits, $\jl_r^{\h\bg_h}$ is an $\einf$-ring.  
\end{proof}

As indicated in the discussion below \eqref{LTD}, our structured higher-homotopical realizations or spectral lifts of the moduli spaces in question bear specific computational tools.  The following is ``dual'' to the $K(h)_\ast$-local $E_h$-Adams spectral sequence studied by Devinatz and Hopkins as a homotopy fixed point spectral sequence \cite[Theorem 1\,(iv)]{devinatz2004homotopy}.  

\begin{proposition}
\label{prop:ss}
    Hypothesis and notations as in Theorem \ref{thm:dual}, there is a natural strongly convergent spectral sequence 
    \begin{equation}
    \label{dual}
        E_2^{s,t}=H^s_\cts\big(\GL_h(\bz_p),\pi_t\ll E_h\big)\implies\pi_{t-s}\bs_{K(h)} 
    \end{equation}
    whose $E^{s,t}_2$-term is the $s$'th continuous cohomology of $\GL_h(\bz_p)$ with coefficients the profinite $\GL_h(\bz_p)$-module $\pi_t\ll E_h$.  Moreover, it fits naturally into a commutative diagram of spectral sequences 
    \[
        \xymatrix{
            & H^s_\cts\big(\GL_h(\bz_p)\times\bg_h,\pi_t\jl\big) \ar@{=>}@[lightgray][ld]_{\hskip-2cm\text{Lubin--Tate}} \ar@{=>}@[brown][rd]^{\hskip.5cm\text{Drinfeld}} & \\
            H^r_\cts(\bg_h,\textcolor{lightgray}{\pi_{t-s}E_h}) \ar@{=>}[rd]_{\hskip-2cm\text{Devinatz--Hopkins}} & & H^r_\cts\big(\GL_h(\bz_p),\textcolor{brown}{\pi_{t-s}\ll E_h}\big) \ar@{=>}[ld]^{\hskip.5cm\eqref{dual}} \\
            & \pi_{t-s-r}\bs_{K(h)} & 
        } 
    \]
    where the colors indicate the abutments of the corresponding spectral sequences.  
\end{proposition}
\begin{proof}
    In view of \eqref{LTD}, we deduce the proposition from the general result that given any profinite group $G$ and a $G$-equivariant spectrum $E$, there is a homotopy fixed point spectral sequence 
    \[
        E_2^{s,t}=H^s_\cts(G,\pi_tE)\implies\pi_{t-s}E^{\h G} 
    \]
    See \cite[Theorem 3.17]{quick2013continuous} and \cite{may1996equivariant} for more details.  
\end{proof}

\begin{remark}
\label{rmk:ss}
    Computationally, in general, the $E_2$-page of \eqref{dual} is expected to be more accessible than that of \eqref{DH} (see discussion immediately following).  This connects to some of the recent progress in $p$-adic geometry, e.g., \cite{Bosco}.  In the same vein, the dual spectrum $\ll E_h$ with $\GL_h(\bz_p)$-action leads to possibilities of finite resolutions of $\bs_{K(h)}$ analogous to those pursued in \cite{Henn,GHMR,Beaudry,Bobkova,BBH}.  
\end{remark}

\subsection{Further problems}
\label{subsec:further}

\subsubsection{Higher homotopy groups of $E_{h,r}$ and $\jl_r$ for $r>0$}
\label{subsubsec:hhg}

In this paper, we defined and studied level structures in the context of spectral algebraic geometry, and obtained from Theorems \ref{thm:ehr} and \ref{thm:jlr} the $\einf$-ring spectra $E_{h,r}$ and $\jl_r$ as global sections of moduli spaces for these derived level structures.  In particular, the level-0 cases $E_{h,0}\simeq\jl_0\simeq E_h$ recover the Morava E-theory spectrum.  Moreover, for all $r\geq0$, $\pi_0E_{h,r}\simeq A_r$ recover Strickland's deformation rings, which play a central role in $E_h$-power operations, and $\pi_0\jl_r$ recover the finite levels of a Lubin--Tate tower.  

As explained in Remarks \ref{rmk:pile} and \ref{rmk:nonlocal}, the homotopy groups $\pi_nE_{h,r}$ and $\pi_n\jl_r$ for $n,r>0$ may encode refined information of arithmetic algebraic geometry, in terms of higher-homotopical coherence of isogenies (instead of isomorphisms).  It would therefore be desirable to develop strategies for quantitative investigation of these invariants.  

In this direction, we propose a Serre-type spectral sequence 
\[
    E_2^{s,t}=H^s(\pi_*E_h,\pi_{*+t}L_{E_{h,r}/E_h})\implies\pi_{t-s}E_{h,r} 
\]
where 
\[
    L_{E_{h,r}/E_h}\simeq\taq^{E_h}(E_{h,r}) 
\]
is the (topological) cotangent complex in \cite{lu-SAG} along the $\einf$-ring map $E_h\to E_{h,r}$ as a globalization of topological Andr\'e--Quillen homology (TAQ).  Conjecturally, the $E_2$-page is inspired by a mapping space spectral sequence for augmented $K(h)$-local commutative $E_h$-algebras studied by Rezk \cite[\S\S2.7 and 2.13]{rezk2013power}, in the context of $E_h$-cohomology of the $K(h)$-localized TAQ.  It is key to compute the cotangent complex of the corresponding moduli problem (cf., e.g., \cite{monavari2024hyperquot}).  The abutment is based on a conceptualization of Behrens and Rezk's modular description related to a comparison between TAQ and the $K(h)$-localized Bousfield--Kuhn functor by examining $E_h$-cohomology of layers of the latter applied to the Goodwillie tower of the identity functor \cite{BKTAQ,mc}.  On a related note, in view of Remark \ref{rmk:nonlocal}, we aim to further investigate the nature of $K(h)$-localization in spectral algebraic geometry.

\subsubsection{Derived level structures and representations}
\label{subsubsec:rep}

Galatius and Venkatesh defined and studied derived Galois deformations in \cite{galatius2018derived}, and compared the action by the homotopy groups of the derived Galois deformation ring and that by the derived Hecke algebra introduced in \cite{venkatesh2019derived} on the (co)homology of locally symmetric spaces (see, in particular, \cite[Definition 5.4\,(iii), Theorem 4.33, Section 7.3, and Theorem 15.2]{galatius2018derived}).  

In Section \ref{subsec:elevel}, we constructed moduli stacks of spectral elliptic curves with derived level structures.  We do not know yet specifically what sorts of Hecke algebras may act on their rings of functions.  In \cite{davies2024hecke,candelori2022topological}, there have been constructions in derived settings of Hecke operators on topological modular forms and Hecke eigenforms, e.g., based on Lurie's theorem \cite[Theorem 1.1]{davies2024hecke}, without involving {\em derived} level structures.  

In view of the close relationship between level structures and representations, it would be interesting to develop a general theory of Hecke algebras in the context of spectral algebraic geometry and find a reasonable construction of derived Hecke stacks compatible with Hecke algebras for topological modular forms.  A natural, related question to investigate is the {\em classicality} of the moduli spaces involved (see, e.g., \cite[Section 3.2.1]{derstr}).

\subsubsection{Relationship to categorical local Langlands correspondence}

From the recently developing categorical perspective, the Langlands correspondence (in its many facets) should be thought of as setting up equivalences between ($\infty$-)categories of certain geometric objects.  See \cite{emerton2022introduction} and the references therein.  For example, for local Langlands correspondence, let $F$ be a finite extension of $\bq_p$, and $\cO$ be the ring of integers in a finite extension of $\bq_\ell$; then the category of smooth $\GL_h(F)$-representations on torsion $\cO$-modules is conjecturally equivalent to a subcategory of quasi-coherent sheaves on a moduli stack $\cx$ which parametrizes $h$-dimensional $\ell$-adic representations of the absolute Galois group $\Gal_F$.  When $\ell\neq p$, such a stack $\cx$ was introduced in \cite{zhu2020coherent} and \cite{dat2024modulilanglandsparameters}.  In \cite{fargues2021geometrization}, this stack is realized by a stack of equivariant vector bundles on the Fargues--Fontaine curve for $F$ (whose \'etale fundamental group is isomorphic to $\Gal_F$).  When $\ell=p$, this stack was constructed in \cite{emerton2022moduli} as a stack of $(\varphi,\Gamma)$-modules.  

Our $\einf$-ring spectrum $\jl$ from Section \ref{subsec:JL} defines a functor 
\[
    \text{$p$-complete spectra}\to\text{$\big(\GL_h(\cO_F)\times D^\times\big)$-equivariant spectra},\quad X\mapsto\jl^X 
\]
where $D/F$ is the central division algebra of invariant $1/h$, and $\jl^X$ denotes the function spectrum of maps $X\to\jl$.  Given the perspective of chromatic homotopy theory \cite{Goerss}, we may ask if this functor can be upgraded to a functor 
\[
    \QCoh(\cm_\fg)\to\text{equivariant sheaves on $\cx$} 
\]
for some suitably constructed moduli stack $\cx$ on the Galois side, where $\cm_\fg$ is the moduli stack of one-dimensional formal groups.  In this case, the categorical local Langlands correspondence may take a form 
\[
    \Mod_{\ll\!E_h}\simeq\IndCoh(\cx) 
\]
between the $\infty$-category of module spectra over the $\GL_h(\cO_F)$-equivariant $\einf$-ring spectrum $\ll E_h$ from Theorem \ref{thm:dual}, and the stable $\infty$-category of Ind-coherent complexes over $\cx$.

\section*{Acknowledgements}

This paper has its impetus from ongoing joint work of Guozhen Wang and the second author, and grew out of the first author's doctoral thesis.  The authors would like to thank William Balderrama, Tobias Barthel, Mark Behrens, Miaofen Chen, Jack Davies, Sanath Devalapurkar, Yanbo Fang, Hui Gao, John Greenlees, Mike Hill, Liang Kong, Igor Kriz, Tyler Lawson, Xiansheng Li, Jiacheng Liang, Fei Liu, Jing Liu, Ruochuan Liu, Peter May, Lennart Meier, Yannan Qiu, Neil Strickland, Akira Tominaga, Guozhen Wang, Yifan Wu, Yuchen Wu, Liang Xiao, Zhouli Xu, Chunshuang Yin, and Xinwen Zhu for helpful discussions and encouragements.  This work was partly supported by National Natural Science Foundation of China grant 12371069 and Guangdong Basic and Applied Basic Research Foundation grant 2023A1515030289.

\bibliographystyle{alpha}
\bibliography{main}

\end{document}